\documentclass[11pt, a4paper]{article}
\usepackage{natbib}
\usepackage{graphicx}
\usepackage{amsmath,amssymb,amsthm,mathtools}
\usepackage{subcaption}
\allowdisplaybreaks[2]
\usepackage{hyperref}
\usepackage{cleveref}
\usepackage{bm}
\usepackage{mathrsfs}
\usepackage{enumitem}
\usepackage[ruled]{algorithm2e}
\usepackage{color}
\usepackage{xcolor}
\numberwithin{equation}{section}
\newtheorem{theorem}{Theorem}[section]
\newtheorem{lemma}[theorem]{Lemma}
\newtheorem{proposition}[theorem]{Proposition}

\newtheorem{remark}[theorem]{Remark}

\newcommand{\curl}{\operatorname{curl}}
\newcommand{\diver}{\operatorname{div}}
\newcommand{\ran}{\operatorname{Range}}

\newcommand{\im}{\operatorname{Im}}
\newcommand{\re}{\operatorname{Re}}
\newcommand{\dd}{\,\mathrm d}
\newcommand{\ii}{\mathrm i}
\newcommand{\R}{\mathbb R}
\newcommand{\GammaR}{\Gamma_R}

\allowdisplaybreaks[4]

\title{A Factorization Method for Support Recovery in Magnetic Induction Tomography}
\author{Junqing Chen\thanks{\footnotesize Corresponding author.
Department of Mathematical Sciences, Tsinghua University, Beijing
100084, China. (jqchen@tsinghua.edu.cn).}
\and Chengzhe Jiang\thanks{\footnotesize
Department of Mathematical Sciences, Tsinghua University, Beijing
100084, China. (jiangcz24@mails.tsinghua.edu.cn)}}

\date{} 

\begin{document}

\maketitle

\begin{abstract}
We develop a factorization framework for magnetic induction tomography in the eddy-current regime. The induced current density is represented by a divergence-free current vector potential, and the external excitation by magnetic dipole densities on an observation sphere. This choice reveals a natural physical factorization of the near-field operator into a data operator, an interior boundary response operator, and an adjoint counterpart. By introducing the Riesz map on the flux space, we recast this factorization in a Hilbert-space setting, which yields a positive-semidefinite operator with a coercive middle factor. Douglas' range theorem then identifies the range of the square root of this dissipative operator with the range of the data operator. For inclusions with a regular boundary, this range identity yields a uniqueness theorem for the support of the conductivity. This characterization leads to a non-iterative spectral indicator based on a regularized Picard quotient computed from the Hermitian imaginary part of the near-field data matrix. The reconstruction stage requires no forward solves. Numerical experiments confirm that the indicator localizes and separates the inclusions accurately, degrades gracefully under increasing noise, and captures non-convex support geometry without any convexity prior.
\end{abstract}
{{\bf Mathematics Subject Classification}(MSC2020): 35R30, 65N21, 78A46}\\
{{\bf Keywords:} magnetic induction tomography, factorization method, inverse eddy current problem, support recovery, non-iterative reconstruction }

\section{Introduction}

Magnetic induction tomography (MIT), also called electromagnetic induction
tomography or eddy-current tomography, is a non-contact imaging modality for
conductive media. An alternating excitation current produces a primary magnetic
field, which induces eddy currents in conductive regions. These currents generate a secondary magnetic field that is measured by the receiver coils. The measured data are then used to recover the conductive structure. The absence of electrodes makes MIT attractive in biomedical imaging and spectroscopy, as well
as in industrial nondestructive testing
\cite{Griffiths2001,ScharfetterCasanasRosell2003,GebauerHankeSchneider2008}.

In general, MIT operates in the low-frequency regime. At such frequencies, the displacement-current term in Maxwell's equations is negligible, leading to the eddy-current approximation. The data are measured away from the unknown conductors.
The resulting inverse
problem is severely ill posed \cite{ChenLiangZou2020}. In practical MIT reconstruction,
one often minimizes a regularized data-misfit functional by iterations. Such iterative
approaches can incorporate prior information and realistic coil models, but
require repeated large-scale forward solves and may be sensitive to
initialization and regularization choices
\cite{ChenLong2024,ScharfetterCasanasRosell2003}.


These costs motivate non-iterative, qualitative methods, which are called direct methods in the literature. They test directly whether a sampling point lies inside the unknown object without reconstructing the full material distribution. One of the first systematic approaches is the linear sampling method, originally proposed by Colton and Kirsch \cite{ColtonKirsch1996} for inverse scattering problems and later extended to electromagnetic inverse scattering \cite{ColtonHaddarMonk2003}. 
Another major development is the factorization method \cite{KirschGrinberg2008}, which exploits the range properties of specific operators to reconstruct the unknown support. It was first established for inverse scattering problems \cite{Kirsch1998}, and was later developed for Maxwell scattering \cite{Kirsch2004}, then was extended to electrical impedance tomography (EIT) problem by Br\"uhl \cite{Bruhl2001} and to low-frequency electromagnetic imaging by Gebauer, Hanke and Schneider \cite{GebauerHankeSchneider2008}.
In addition, there are other direct methods in which an indicator is constructed by correlating measured data with suitably chosen probing functions. Due to the extensive literature on these methods, we list only a few representative works: the reverse time migration for inverse acoustic scattering problems \cite{Chen_2013}, and the direct sampling method for inverse electromagnetic scattering problems \cite{ItoJinZou2013} and EIT \cite{ChowItoZou2014}.

For MIT problems specifically, non-iterative imaging is less
developed, but several distinct lines of work exist. For small conductive
targets, Ammari and coauthors derived asymptotic formulas and MUSIC-type
location algorithms based on conductive polarization tensors
\cite{AmmariChenChenVolkovWang2013}, and Haddar and Riahi developed a near-field
linear sampling method for axisymmetric configurations \cite{HaddarRiahi2021}.
Arnold and Harrach proved unique shape detection for transient eddy currents and
related their constructive criterion to the factorization method
\cite{ArnoldHarrach2013}. We introduced a direct sampling method for MIT in our recent work
\cite{ChenJiang}. The most extensively developed framework is the
monotonicity principle method, introduced for electrical resistance tomography
by Tamburrino and Rubinacci \cite{TamburrinoRubinacci2002} and carried to the
magneto-quasi-stationary regime through the monotonicity of the source-free
time constants \cite{TamburrinoPiscitelliZhou2021} and of a Laplace-domain transfer operator \cite{TamburrinoCorboPiscitelli2026}. Despite this activity, the factorization method has not, to the best of our knowledge, been established for the time-harmonic eddy-current near-field operator of MIT. 

This paper aims to address this gap. We develop a factorization framework for support recovery in a time-harmonic eddy-current MIT model with constant magnetic permeability and conductivity
supported in an unknown bounded region \(D\). The key observation is that the
induced current \(J=\sigma E\) admits a current vector potential \(M\) with
\(J=\curl M\), \(\diver M=0\), and \(M\times n=0\) on \(\partial D\). This
reduces the full-space problem to an interior variational problem in \(D\), in
which the exterior induced field depends only on the normal flux \(M\cdot n\).
We introduce the flux space \(X_D\), the data operator \(G_D\), and the boundary
response operator \(B_D\), and obtain the physical factorization of the near-field operator
\(F_D=G_D B_D G_D^*\). Passing through the Riesz map on \(X_D\) to separate the
dual adjoint from the Hilbert-space adjoint, and using the coercivity of the
dissipative part of the Riesz-lifted response, we obtain the positive
factorization \(\im F_D = G_D P_D G_D^\dagger\). Douglas' range theorem then
gives the central identity
\[
\ran\bigl((\im F_D)^{1/2}\bigr)=\ran(G_D).
\]

For inclusions whose connected components are either of class \(C^{1,1}\) or
convex, this turns into the pointwise characterization
\[
a\in D
\quad\Longleftrightarrow\quad
\varphi_{a,p}\in\ran\bigl((\im F_D)^{1/2}\bigr),
\]
for dipole sampling fields \(\varphi_{a,p}\), so that the conductive support is
uniquely determined by the magnetic near-field operator. This characterization
suggests a non-iterative spectral indicator. We derive its finite-dimensional realization and validate it on simulated data.

The paper is organized as follows. Section \ref{sect2} formulates the full-space
eddy-current model, derives the current-potential formulation, and proves the
well-posedness of the interior variational problem. Section \ref{sect3} constructs the
near-field operator and proves the factorization range identity. These two
sections require only that \(D\) be a bounded Lipschitz open set that
satisfies the topological conditions described in the following. Section \ref{sect4} uses the factorization
identity to characterize and uniquely determine the support for domains 
whose connected components are either of class \(C^{1,1}\) or convex. 
Section \ref{sect5} presents the discrete spectral indicator and the numerical experiments.
Finally, section \ref{sect6} concludes the paper. 
\section{Problem setup} \label{sect2}

\subsection{Geometric assumptions and functional spaces}

Let \(D\Subset B_R\subset\R^3\) be the conducting region, possibly with finitely
many connected components \(D_1,\ldots,D_N\), and let
\[
\GammaR:=\partial B_R
\]
be the measurement and excitation sphere. Throughout Sections~2 and~3 we assume
that each \(D_j\) is a bounded Lipschitz open set, is simply connected, and has
connected boundary. The connected-boundary assumption excludes insulating
cavities, and simple connectedness excludes handle-type components. 
Additional regularity properties (P1)--(P2) satisfied by
\(C^{1,1}\) and convex domains are imposed only in Section~4.

We assume that the conductivity $\sigma\in L^\infty(\mathbb R^3)$ is supported in \(D\):
\[
\sigma(x)=0\quad\text{in }\R^3\setminus\overline D,
\qquad
0<\sigma_0\leq \sigma(x)\leq \sigma_1<\infty
\quad\text{in }D.
\]
The magnetic permeability \(\mu>0\) is constant in all of \(\R^3\). We work with
the time convention \(e^{-\ii\omega t}\), where \(\omega>0\). All function
spaces are complex. We use Hilbert-space inner products that are linear in the
first argument and antilinear in the second. Dual spaces are anti-duals. Thus
\(X'\) denotes the space of continuous anti-linear functionals on \(X\).

We use the standard Hilbert spaces of vector fields 
with square-integrable curl and divergence, respectively
\[
\begin{aligned}
H(\curl,D):=&\left\{u\in L^2(D)^3:\curl u\in L^2(D)^3\right\},
\qquad\\
H(\diver,D):=&\left\{u\in L^2(D)^3:\diver u\in L^2(D)\right\},
\end{aligned}
\]
and we write \(H_0(\curl,D)\) for the subspace of \(H(\curl,D)\) whose elements
have vanishing tangential trace \(u\times n=0\) on \(\partial D\).

The current vector potential will be sought in
\begin{equation}
A_D:=
\left\{
M\in H_0(\curl,D)\cap H(\diver,D):
\diver M=0
\right\}.
\label{eq:AD}
\end{equation}
We equip \(A_D\) with the
norm
\[
\|M\|_{A_D}^2
:=
\|M\|_{L^2(D)^3}^2+
\|\curl M\|_{L^2(D)^3}^2.
\]

The above topological assumptions are precisely what we need for the vector
potential representation used below. First, the Dirichlet harmonic field space
is trivial:
\begin{equation}
\mathcal H_D:=
\left\{
U\in H_0(\curl,D):
\curl U=0,\ \diver U=0
\right\}
=\{0\}.
\label{eq:harmonic-field-zero}
\end{equation}
Since every \(M\in A_D\) is divergence free, this yields \cite{Monk2003}
\begin{equation}
\|M\|_{A_D}
\leq C\|\curl M\|_{L^2(D)^3},
\qquad
M\in A_D.
\label{eq:Maxwell-estimate}
\end{equation}
Second, the same topology yields the uniqueness of vector 
potential \cite[Theorem~3.17]{ABDG1998}. Every \(J\in H(\diver,D)\) satisfying
\begin{equation}
\diver J=0\quad\text{in }D,
\qquad
J\cdot n=0\quad\text{on }\partial D
\label{eq:J-constraints}
\end{equation}
admits a unique potential \(M\in A_D\) such that
\begin{equation}
J=\curl M.
\label{eq:J-curl-M}
\end{equation}
Conversely, every \(M\in A_D\) gives \(J=\curl M\) satisfying
\eqref{eq:J-constraints}. Thus \(A_D\) is the natural space for the vector
potential.

Define the normal flux space
\begin{equation}
X_D:=\gamma_n(A_D)
=
\{M\cdot n:\ M\in A_D\},
\label{eq:XD-definition}
\end{equation}
equipped with the quotient norm
\begin{equation}
\|q\|_{X_D}
:=
\inf_{\substack{M\in A_D\\ M\cdot n=q}}
\|M\|_{A_D}.
\label{eq:XD-norm}
\end{equation}
Since \(X_D\) is isometrically isomorphic to the Hilbert quotient
\(A_D/\ker\gamma_n\), it is a Hilbert space. For a Lipschitz domain \(D\), one always has
a continuous embedding
\[
X_D\hookrightarrow H^{-1/2}(\partial D),
\]
but \(X_D\) does not need to have a simple Sobolev-space characterization.

\subsection{The eddy-current model for MIT}

Following the standard eddy-current formulation used in electromagnetic
induction imaging, see, for instance, Ammari et al.~\cite{AmmariChenChenVolkovWang2013},
the total fields \((E,H)\) in the presence of \(D\) solve the full-space problem
\begin{equation}
\left\{
\begin{aligned}
\curl E &=\ii\omega\mu H
&&\text{in }\R^3,\\
\curl H &=J_0+\sigma E
&&\text{in }\R^3,\\
E(x)&=O(|x|^{-1}),\qquad H(x)=O(|x|^{-1})
&&\text{as }|x|\to\infty.
\end{aligned}
\right.
\label{eq:full-eddy}
\end{equation}
Here \(J_0\) is a divergence-free source current with compact support
disjoint from \(\overline D\). The decay condition in \eqref{eq:full-eddy} fixes the physical
full-space solution. Equivalently, eliminating \(H\) gives
\[
\curl\mu^{-1}\curl E-\ii\omega\sigma E=\ii\omega J_0
\quad\text{in }\R^3,
\qquad
\diver E=0\quad\text{in }\R^3\setminus\overline D,
\]
together with \(E(x)=O(|x|^{-1})\) at infinity. 

Let \((E_0,H_0)\) be the background field obtained by setting \(\sigma=0\):
\[
\curl E_0=\ii\omega\mu H_0,\qquad
\curl H_0=J_0
\quad\text{in }\R^3,
\qquad
E_0,H_0=O(|x|^{-1}).
\]
Since \(J_0\) is supported outside \(D\), the incident magnetic field satisfies
\[
\curl H_0=0,\qquad \diver H_0=0
\quad\text{in a neighbourhood of }\overline D.
\]
By the simple connectedness of each \(D_j\), the incident magnetic field can be
written as
\[
H_0=\nabla h,\qquad \Delta h=0\quad\text{in }D.
\]
This is the incident-field class used in the measurement model below, where
fields are generated by magnetic dipole potentials on \(\GammaR\).

The induced current density is
\[
J:=\sigma E.
\]
Since \(\sigma=0\) outside \(D\), \(J\) vanishes identically on
\(\R^3\setminus\overline D\). From now on we regard it as a function supported
in \(D\). Taking the divergence of the second equation in \eqref{eq:full-eddy}
and using \(\diver J_0=0\) gives \eqref{eq:J-constraints}. By the
vector-potential result stated above, there is a unique \(M\in A_D\) such that
\[
J=\curl M.
\]

Let
\[
\Phi(x,y)=\frac{1}{4\pi |x-y|}
\]
be the fundamental solution of \(-\Delta\). By the Biot-Savart law, the
induced magnetic field generated by \(J\) is
\begin{equation}
H-H_0=\curl S_DJ,
\qquad
S_DJ(x):=\int_D\Phi(x,y)J(y)\,\dd y.
\label{eq:Biot-Savart-J}
\end{equation}
Here \(S_D\) is the Newtonian volume potential
\[
S_D:L^2(D)^3\to H^2_{\mathrm{loc}}(\R^3)^3,
\]
which satisfies \(-\Delta S_D u=u\) in \(D\) and \(\Delta S_D u=0\) in
\(\R^3\setminus\overline D\); the same symbol is used for its action on scalar
fields.
For \(J=\curl M\) with \(M\in A_D\), the boundary condition \(M\times n=0\)
gives
\[
S_DJ=S_D(\curl M)=\curl S_D M
\quad\text{in }\R^3\setminus\partial D,
\]
by integration by parts. Hence
\[
H-H_0=\curl\curl S_D M.
\]
Using \(E=\rho\curl M\) in \(D\), \(\curl E=\ii\omega\mu H\), and
\(H=H_0+\curl\curl S_D M\), we obtain
\begin{equation}
\frac{1}{\ii\omega\mu}
\curl\left(\rho\,\curl M\right)
-
\curl\curl S_D M
=
\nabla h
\quad\text{in }D,
\label{eq:M-strong}
\end{equation}
where
\[
\rho:=\sigma^{-1},
\qquad
S_D M(x):=\int_D\Phi(x,y)M(y)\,\dd y.
\]

\subsection{A boundary representation of the nonlocal term}

We record a useful identity for the nonlocal term. Recall the convention
\[
\Phi(x,y)=\frac{1}{4\pi |x-y|},
\qquad
-\Delta_x\Phi(x,y)=\delta_y(x).
\]
For \(M\in A_D\), set $q_M:=M\cdot n\in X_D$ and let
\begin{equation}
S_{\partial D}{q_M}(x):=
\int_{\partial D}\Phi(x,y)q_M(y)\,\dd S_y.
\label{eq:singlelayer}
\end{equation}
$S_{\partial D}$ is just the single layer potential operator. Then, for \(x\in D\),
\begin{equation}
\diver_x S_D M(x)=-S_{\partial D}{q_M}(x).
\label{eq:div-SM}
\end{equation}
Indeed, using \(\nabla_x\Phi(x,y)=-\nabla_y\Phi(x,y)\) and integration by parts,
\[
\begin{aligned}
\diver_x S_D M(x)
&=
\int_D \nabla_x\Phi(x,y)\cdot M(y)\,\dd y
\\
&=
-\int_D \nabla_y\Phi(x,y)\cdot M(y)\,\dd y
\\
&=
-\int_{\partial D}\Phi(x,y)\bigl(M(y)\cdot n(y)\bigr)\,\dd S_y
+
\int_D \Phi(x,y)\,\diver M(y)\,\dd y
\\
&=
-S_{\partial D}{q_M}(x),
\end{aligned}
\]
where we used \(\diver M=0\) in \(D\) in the last step. Moreover,
\[
\Delta_x S_D M(x)=-M(x)
\quad\text{in }D.
\]
Therefore,
\begin{equation}
\curl\curl S_D M
=
\nabla\diver S_D M-\Delta S_D M
=
-\nabla S_{\partial D}{q_M}+M
\quad\text{in }D.
\label{eq:curlcurl-SM}
\end{equation}

Let \(V\in A_D\), and set \(q_V:=V\cdot n\). Since \(V\times n=0\) on
\(\partial D\), integration by parts gives
\[
(\curl S_D M,\curl V)_D
=
(\curl\curl S_D M,V)_D.
\]
Using \eqref{eq:curlcurl-SM} and \(\diver V=0\), we obtain
\[
\begin{aligned}
(\curl S_D M,\curl V)_D
&=
(M,V)_D
-
(\nabla S_{\partial D}{q_M},V)_D
\\
&=
(M,V)_D
-
\int_{\partial D}S_{\partial D}{q_M}(x)\overline{q_V(x)}\,\dd S_x.
\end{aligned}
\]
Equivalently,
\begin{equation}
(\curl S_D M,\curl V)_D
=
(M,V)_D
-
\mathcal V_D(q_M,q_V),
\label{eq:nonlocal-identity}
\end{equation}
where
\begin{equation}
\mathcal V_D(q,r)
:=
\int_{\partial D}
\left(
\int_{\partial D}\Phi(x,y)q(y)\,\dd S_y
\right)
\overline{r(x)}\,\dd S_x.
\label{eq:single-layer-form}
\end{equation}
In particular,
\begin{equation}
(\curl S_D M,\curl M)_D
=
\|M\|_{L^2(D)}^2-\mathcal V_D(q_M,q_M)
\in \mathbb R.
\label{eq:nonlocal-real}
\end{equation}
Since \(\Phi(x,y)\) is real-valued, the single-layer form
\(\mathcal V_D(q,r)\) is Hermitian. Consequently \(\mathcal V_D(q,q)\) is real for any \(q\in X_D\).
For the coercivity argument below, only the reality of this quantity is needed.

\subsection{Weak formulation and well-posedness}

For \(M,V\in A_D\), define
\begin{equation}
a_D(M,V)
:=
\frac{1}{\ii\omega\mu}
\int_D
\rho\,\curl M\cdot \overline{\curl V}\,dx
-
\int_D
\curl S_D M\cdot \overline{\curl V}\,dx.
\label{eq:aD-original}
\end{equation}
Using \eqref{eq:nonlocal-identity}, this form can equivalently be written as
\begin{equation}
a_D(M,V)
=
\frac{1}{\ii\omega\mu}
(\rho\curl M,\curl V)_D
-
(M,V)_D
+
\mathcal V_D(M\cdot n,V\cdot n).
\label{eq:aD-boundary-representation}
\end{equation}
Here
\[
(u,v)_D:=\int_D u\cdot \overline v\,dx,
\]
and \(\mathcal V_D\) is the single-layer boundary form in
\eqref{eq:single-layer-form}. The integrals in
\eqref{eq:single-layer-form} are understood as the natural
\(H^{1/2}(\partial D)\)-\(H^{-1/2}(\partial D)\) duality pairing.

For the incident fields used below, let \(h\) be the magnetic scalar potential generated
by a magnetic dipole whose pole lies outside \(\overline D\). Then \(h\) is
harmonic in \(D\) and smooth in a neighbourhood of \(\overline D\). Its trace on
\(\partial D\) defines an anti-linear boundary functional
\[
\ell_h(q):=\int_{\partial D}h\,\overline q\,\dd S,
\qquad q\in X_D,
\]
which is bounded with respect to the quotient norm of \(X_D\). Indeed, for
any \(q\in X_D\) and any \(M_q\in A_D\) with \(M_q\cdot n=q\), Green's formula
and \(\diver M_q=0\) give
\[
\ell_h(q)
=
\int_{\partial D}h\,\overline{M_q\cdot n}\,\dd S
=
(\nabla h,M_q)_D.
\]
Hence
\[
|\ell_h(q)|
\leq
\|\nabla h\|_{L^2(D)^3}\|M_q\|_{A_D}.
\]
Taking the infimum over all such liftings \(M_q\) yields
\[
|\ell_h(q)|
\leq
\|\nabla h\|_{L^2(D)^3}\|q\|_{X_D}.
\]
Thus \(\ell_h\in X_D'\). Now let \(R_D:X_D\to X_D'\) denote the Riesz map
associated with the quotient Hilbert structure of \(X_D\), i.e.
\[
\langle R_D f,q\rangle_{X_D',X_D}=(f,q)_{X_D},
\qquad f,q\in X_D.
\]
By the Riesz representation theorem, there exists a unique
\(f=R_D^{-1}\ell_h\in X_D\) such that
\begin{equation}
\ell_h(q)=(f,q)_{X_D},
\qquad q\in X_D. \label{f_def}
\end{equation}
With this convention, for a given \(h\), the weak problem to \eqref{eq:M-strong} is to find
\(M_f\in A_D\) such that
\begin{equation}
a_D(M_f,V)=(f,V\cdot n)_{X_D},
\qquad
\forall V\in A_D,
\label{eq:weak-f}
\end{equation}
where $f\in X_D$ is related to $h$ by \eqref{f_def}.

\begin{proposition}[Well-posedness]
Assume that \(D\) is a bounded Lipschitz open set with finitely many connected
components, each of which is simply connected and has connected boundary, and
that \(0<\sigma_0\leq \sigma\leq \sigma_1<\infty\).
Then for every \(f\in X_D\), problem \eqref{eq:weak-f} admits a unique
solution \(M_f\in A_D\). Moreover,
\[
\|M_f\|_{A_D}
\leq C\|f\|_{X_D}.
\]
\end{proposition}

\begin{proof}
The boundedness of \(a_D\) follows from the boundedness of \(\rho\), the
continuity of \(S_D\), and the continuity of the normal trace
\[
A_D\ni M\mapsto M\cdot n\in X_D.
\]
Alternatively, using \eqref{eq:aD-boundary-representation}, the boundedness of
the single-layer form \(\mathcal V_D\) on \(X_D\times X_D\) follows from the
continuous embedding \(X_D\hookrightarrow H^{-1/2}(\partial D)\) and the
mapping property of the Laplace single-layer operator.

We next prove coercivity after a rotation. For \(M\in A_D\), by
\eqref{eq:nonlocal-real},
\[
(\curl S_D M,\curl M)_D\in\mathbb R.
\]
Therefore
\[
\begin{aligned}
\operatorname{Re}\bigl(\ii a_D(M,M)\bigr)
&=
\operatorname{Re}
\left[
\frac{1}{\omega\mu}
\int_D \rho|\curl M|^2\,dx
-
\ii(\curl S_D M,\curl M)_D
\right]
\\
&=
\frac{1}{\omega\mu}
\int_D \rho|\curl M|^2\,dx
\\
&\geq
\frac{1}{\omega\mu\sigma_1}
\|\curl M\|_{L^2(D)}^2.
\end{aligned}
\]
By the Maxwell estimate on \(A_D\), using \(\mathcal H_D=\{0\}\),
\[
\|M\|_{A_D}\leq C\|\curl M\|_{L^2(D)}.
\]
Hence
\[
\operatorname{Re}\bigl(\ii a_D(M,M)\bigr)
\geq c\|M\|_{A_D}^2.
\]
The rotated form
\[
\widetilde a_D(M,V):=\ii a_D(M,V)
\]
is coercive on \(A_D\).

The right-hand side
\[
V\mapsto \ii(f,V\cdot n)_{X_D}
\]
is bounded on \(A_D\), because
\[
|(f,V\cdot n)_{X_D}|
\leq
\|f\|_{X_D}\|V\cdot n\|_{X_D}
\leq
\|f\|_{X_D}\|V\|_{A_D}.
\]
The Lax--Milgram theorem applied to \(\widetilde a_D\) gives a unique
\(M_f\in A_D\) satisfying
\[
\ii a_D(M_f,V)
=
\ii(f,V\cdot n)_{X_D},
\qquad
\forall V\in A_D.
\]
Equivalently, \(M_f\) solves \eqref{eq:weak-f}. The stability estimate
follows from coercivity and boundedness.
\end{proof}

\section{Factorization of the near-field operator} \label{sect3}

In this section, we construct the magnetic-dipole near-field operator and prove
the factorization range identity. Throughout, the data are measured on \(\GammaR=\partial B_R\), where $B_R$ denotes a sphere of radius R. We denote the measurement space by
\[
Y:=L^2(\GammaR)^3.
\]
Two adjoints of the data operator appear. The
symbol \(G_D^*\) denotes the \emph{anti-dual adjoint}
\[
G_D^*:Y\to X_D',
\]
which represents the boundary functional induced by the free-space 
magnetic scalar potential. The symbol \(G_D^\dagger\) denotes the \emph{Hilbert-space adjoint}
\[
G_D^\dagger:Y\to X_D
\]
with respect to the quotient Hilbert structure of \(X_D\). These two adjoints
are related by the Riesz map on \(X_D\).

\subsection{Data operator and magnetic dipole sources}

For \(q\in X_D\), using the Biot-Savart law \eqref{eq:Biot-Savart-J} and the definition of single layer potential \eqref{eq:singlelayer}, 
the corresponding induced magnetic field is
\[
H_s(x)=-\nabla S_{\partial D}  q(x),\quad x\in \R^3\setminus\overline D.
\]
We define the data operator into the measurement space \(Y\) by
\begin{equation}
G_D:X_D\to Y,
\qquad
G_Dq:=-\nabla S_{\partial D} q|_{\GammaR}.
\label{eq:GD}
\end{equation}
Since \(\operatorname{dist}(\partial D,\GammaR)>0\), the kernel is smooth on
\(\GammaR\times \partial D\). Together with the continuous embedding
\(X_D\hookrightarrow H^{-1/2}(\partial D)\), this implies that
\(G_D:X_D\to Y\) is bounded and compact.

Let \(m\in Y\) be a magnetic dipole density on \(\GammaR\). It generates in
\(D\) the free-space magnetic scalar potential
\begin{equation}
h_m(y):=
\int_{\GammaR} m(x)\cdot \nabla_y\Phi(x,y)\,\dd S_x,
\qquad y\in D.
\label{eq:hm}
\end{equation}
Then \(h_m\) is harmonic in a neighbourhood of \(\overline D\). For
\(q\in X_D\), using the convention that inner products are linear in the first
argument, we compute
\[
\begin{aligned}
(m,G_Dq)_Y
&=
-\int_{\GammaR}
 m(x)\cdot
 \overline{\nabla_x\int_{\partial D}\Phi(x,y)q(y)\,\dd S_y}
\,\dd S_x
\\
&=
-\int_{\partial D}
\left[
\int_{\GammaR}m(x)\cdot \nabla_x\Phi(x,y)\,\dd S_x
\right]\overline{q(y)}\,\dd S_y
\\
&=
\int_{\partial D}
\left[
\int_{\GammaR}m(x)\cdot \nabla_y\Phi(x,y)\,\dd S_x
\right]\overline{q(y)}\,\dd S_y
\\
&=
\int_{\partial D} h_m(y)\overline{q(y)}\,\dd S_y.
\end{aligned}
\]
This identity shows that the anti-dual adjoint of \(G_D\) is the operator
\begin{equation}
G_D^*:Y\to X_D',
\qquad
\langle G_D^*m,q\rangle_{X_D',X_D}:=(m,G_Dq)_Y,
\label{eq:dual-adjoint-G}
\end{equation}
and that
\begin{equation}
G_D^*m=\ell_{h_m},
\qquad
\ell_{h_m}(q):=\int_{\partial D}h_m\overline q\,\dd S.
\label{eq:GD-star-physical}
\end{equation}
Thus \(G_D^*m\) is the boundary functional induced by the physical 
free-space magnetic scalar potential \(h_m\). It should not be confused with
the Hilbert adjoint of \(G_D\).

Let
\[
R_D:X_D\to X_D'
\]
be the Riesz map of the Hilbert space \(X_D\), i.e.
\[
\langle R_D f,q\rangle_{X_D',X_D}=(f,q)_{X_D},
\qquad f,q\in X_D.
\]
The Hilbert adjoint \(G_D^\dagger:Y\to X_D\) is characterized by
\begin{equation}
(G_D^\dagger m,q)_{X_D}=(m,G_Dq)_Y,
\qquad m\in Y,
\ q\in X_D.
\label{eq:Hilbert-adjoint-G}
\end{equation}
Comparing \eqref{eq:dual-adjoint-G} and \eqref{eq:Hilbert-adjoint-G}, we obtain
\begin{equation}
G_D^*=R_DG_D^\dagger.
\label{eq:dual-Hilbert-adjoint-relation}
\end{equation}
Equivalently,
\[
G_D^\dagger m=R_D^{-1}\ell_{h_m}.
\]
This element is the \(X_D\)-Riesz representative of the physical boundary
functional \(\ell_{h_m}\). It should not be identified with the function
\(h_m|_{\partial D}\) itself.

\subsection{Boundary response operators}

We first define the physical boundary response operator with its natural
input space. For \(\ell\in X_D'\), let \(M_\ell\in A_D\) be the unique solution
of
\begin{equation}
a_D(M_\ell,V)=\langle \ell,V\cdot n\rangle_{X_D',X_D},
\qquad
\forall V\in A_D.
\label{eq:weak-ell-factorization}
\end{equation}
Then define
\begin{equation}
B_D:X_D'\to X_D,
\qquad
B_D\ell:=M_\ell\cdot n.
\label{eq:BD-dual}
\end{equation}
The boundedness of \(B_D\) follows from the well-posedness of the variational
problem.

For a magnetic dipole density \(m\), the input functional is precisely
\(G_D^*m=\ell_{h_m}\). Since
\[
\langle \ell_{h_m},V\cdot n\rangle
=
\int_{\partial D}h_m\overline{V\cdot n}\,\dd S
=
(\nabla h_m,V)_D,
\qquad V\in A_D,
\]
where \(\diver V=0\) has been used, the solution \(M_{G_D^*m}\) of
\eqref{eq:weak-ell-factorization} is the current vector potential
corresponding to the incident magnetic field \(H_0=\nabla h_m\). Its boundary
flux is
\[
B_DG_D^*m=M_{G_D^*m}\cdot n,
\]
and the measured induced magnetic field is obtained by applying \(G_D\).
Consequently, the physical magnetic-dipole near-field operator is
\begin{equation}
F_D:Y\to Y,
\qquad
F_D=G_DB_DG_D^*.
\label{eq:F-factorization}
\end{equation}
The spaces in this factorization are
\[
Y\xrightarrow{\,G_D^*\,}X_D'
\xrightarrow{\,B_D\,}X_D
\xrightarrow{\,G_D\,}Y.
\]

To apply a Hilbert-space range theorem, we now use the Riesz map to rewrite the
middle response as an operator on \(X_D\). Define
\begin{equation}
\mathcal B_D:=B_DR_D:X_D\to X_D.
\label{eq:BD-hilbertized}
\end{equation}
Equivalently, for \(f\in X_D\), \(\mathcal B_Df=M_f\cdot n\), where
\(M_f\in A_D\) is the solution of \eqref{eq:weak-f}.
Using \eqref{eq:dual-Hilbert-adjoint-relation}, the physical factorization
\eqref{eq:F-factorization} becomes
\begin{equation}
F_D=G_D\mathcal B_DG_D^\dagger.
\label{eq:F-Hilbert-factorization}
\end{equation}
Here \(G_D^\dagger\) is the Hilbert adjoint of \(G_D\).

\subsection{Coercivity of the Riesz-lifted response}

For an operator \(T:X_D\to X_D\), we use the convention
\[
\im T:=\frac{T-T^\dagger}{2\ii},
\]
so that, since our inner products are linear in the first argument,
\[
\begin{aligned}
(f,(\im T)f)_{X_D}=&-\frac1{2\ii}\left((f,Tf)_{X_D}-(f,T^\dagger f)_{X_D}\right)\\
=&-\im(f,Tf)_{X_D}.
\end{aligned}
\]
Define
\begin{equation}
P_D:=\im\mathcal B_D:X_D\to X_D.
\label{eq:PD-definition}
\end{equation}

\begin{lemma}[Coercivity of the dissipative part]
The operator \(P_D=\im\mathcal B_D\) is bounded, self-adjoint, positive, and
coercive on \(X_D\). More precisely, there is a constant \(c>0\) such that
\begin{equation}
(f,P_Df)_{X_D}
\ge c\|f\|_{X_D}^2,
\qquad
\forall f\in X_D.
\label{eq:PD-coercive}
\end{equation}
\end{lemma}

\begin{proof}
Let \(M=M_f\), where \(M_f\) solves \eqref{eq:weak-f}. Since
\(\mathcal B_Df=M\cdot n\), taking \(V=M\) in \eqref{eq:weak-f} gives
\[
(f,\mathcal B_Df)_{X_D}=a_D(M,M).
\]
By \eqref{eq:aD-boundary-representation},
\[
a_D(M,M)
=
-\frac{\ii}{\omega\mu}
\int_D \rho|\curl M|^2\,dx
-
\|M\|_{L^2(D)}^2
+
\mathcal V_D(M\cdot n,M\cdot n).
\]
The last two terms are real. Therefore
\begin{equation}
(f,P_Df)_{X_D}
=
-\im(f,\mathcal B_Df)_{X_D}
=
\frac{1}{\omega\mu}
\int_D\rho|\curl M|^2\,dx.
\label{eq:PD-energy-identity}
\end{equation}
Using \(\rho\ge \sigma_1^{-1}\) and the Maxwell estimate
\eqref{eq:Maxwell-estimate}, we obtain
\[
(f,P_Df)_{X_D}
\ge c_0\|M\|_{A_D}^2.
\]
It remains to control \(\|f\|_{X_D}\) by \(\|M\|_{A_D}\). From the weak problem
and the boundedness of \(a_D\),
\[
|(f,V\cdot n)_{X_D}|
=
|a_D(M,V)|
\le C\|M\|_{A_D}\|V\|_{A_D},
\qquad
\forall V\in A_D.
\]
Taking the infimum over all \(V\in A_D\) with \(V\cdot n=q\) gives
\[
|(f,q)_{X_D}|
\le C\|M\|_{A_D}\|q\|_{X_D},
\qquad
\forall q\in X_D.
\]
Hence
\[
\|f\|_{X_D}
=
\sup_{q\ne0}\frac{|(f,q)_{X_D}|}{\|q\|_{X_D}}
\le C\|M\|_{A_D}.
\]
Combining this with \eqref{eq:PD-energy-identity} proves
\eqref{eq:PD-coercive}. Since \(P_D=\im\mathcal B_D\) is self-adjoint by
definition, it is a bounded, positive, coercive operator on \(X_D\).
\end{proof}

\subsection{The range identity}

From \eqref{eq:F-Hilbert-factorization} and the definition of \(P_D\), taking
imaginary parts in Hilbert space \(Y\) gives
\begin{equation}
\im F_D=G_D P_D G_D^\dagger.
\label{eq:positive-factorization}
\end{equation}
Thus \(\im F_D\) is semi-positive. Indeed, since \(F_D=G_D\mathcal B_DG_D^\dagger\), we have
\[
F_D^\dagger=G_D\mathcal B_D^\dagger G_D^\dagger,
\]
which yields
\[
\im F_D
=
G_D\left(\frac{\mathcal B_D-\mathcal B_D^\dagger}{2\ii}\right)G_D^\dagger
=
G_D(\im\mathcal B_D)G_D^\dagger
=
G_D P_DG_D^\dagger.
\]
Since \(P_D\ge cI\), the square root \(P_D^{1/2}:X_D\to X_D\) is boundedly
invertible. Set
\[
L_D:=G_DP_D^{1/2}:X_D\to Y.
\]
Because \(G_D^\dagger\) is the Hilbert adjoint of \(G_D\) and
\(P_D^{1/2}\) is self-adjoint,
\[
L_D^\dagger=(G_DP_D^{1/2})^\dagger=P_D^{1/2}G_D^\dagger.
\]
Therefore \eqref{eq:positive-factorization} gives
\[
\im F_D=L_DL_D^\dagger.
\]
By Douglas' range theorem \cite{Douglas1966}, for every bounded operator
\(L:X\to Y\) between Hilbert spaces,
\[
\ran\bigl((LL^\dagger)^{1/2}\bigr)=\ran(L).
\]
Applying this to \(L=L_D\), we find
\[
\ran\bigl((\im F_D)^{1/2}\bigr)
=
\ran(L_D).
\]
Finally, since \(P_D^{1/2}\) is an isomorphism of \(X_D\) onto itself,
\[
\ran(L_D)
=
\ran(G_DP_D^{1/2})
=
\ran(G_D).
\]
We have proved the following result.

\begin{theorem}[Factorization range identity]
Let \(D\) be a bounded Lipschitz open set with finitely many connected
components, each of which is simply connected and has connected boundary, and
let \(F_D\) be the magnetic-dipole near-field operator defined by
\eqref{eq:F-factorization}. Then
\begin{equation}
\ran\bigl((\im F_D)^{1/2}\bigr)=\ran(G_D).
\label{eq:range-identity}
\end{equation}
\end{theorem}


\section{Support uniqueness}\label{sect4}

This section proves that, when each connected component of $D$
is either of class $C^{1,1}$ or convex, the inclusion $D$ is uniquely determined by the
magnetic-dipole near-field operator $F_D$. 
The proof rests on two analytic properties that we isolate as (P1)--(P2) below,
which both geometric classes are shown to satisfy. For the pointwise
characterization we also assume that $B_R\setminus\overline D$ is connected.
Let $D_1,\ldots,D_N$ be the connected components of $D$.

\subsection{Two regularity properties and the admissible geometries}

In addition to the topological assumptions of Section~2, we further assume
throughout this section that $D$ has the following two properties.
\begin{enumerate}[label=(P\arabic*),leftmargin=*]
\item Maxwell regularity: $A_D\subset H^1(D)^3$, with continuous embedding.
\item $H^2$ Dirichlet regularity: whenever $g$ is the trace on $\partial D$ of a
function smooth in a neighbourhood of $D$, the solution $u\in H^1(D)$ of
$\Delta u=0$ in $D$, $u=g$ on $\partial D$, belongs to $H^2(D)$.
\end{enumerate}

\begin{lemma}[Admissible geometries]\label{lem:admissible}
Suppose each connected component of $D$ is either of class $C^{1,1}$ or convex.
Then $D$ satisfies (P1) and (P2).
\end{lemma}

\begin{proof}
Both properties can be verified componentwise. Fix a connected component $\Omega$ of $D$.
For (P1), the embedding $H_0(\curl,\Omega)\cap H(\diver,\Omega)\subset
H^1(\Omega)^3$ holds for $C^{1,1}$ domains by \cite[Theorem~2.12]{ABDG1998}
and for convex domains by \cite[Theorem~2.17]{ABDG1998}; restricting to
divergence-free fields gives $A_\Omega\subset H^1(\Omega)^3$.
For (P2), the elliptic shift estimate is classical for $C^{1,1}$ domains
\cite[Theorem~2.3.3.2]{Grisvard1985} and holds for convex domains by
\cite[Theorem~3.2.1.2]{Grisvard1985}.
\end{proof}

\begin{remark}
The two classes cover different geometries, and neither contains the other. $C^{1,1}$
regularity allows non-convex but smooth inclusions, whereas convexity allows
inclusions with edges and corners, such as the cubes used in Section~5.
\end{remark}

\subsection{The flux space}

Under (P1) the abstract flux space $X_D$ of Section~2 admits a concrete
description. Recall that elements of $A_D$ have vanishing tangential trace,
so their boundary trace is carried entirely by the normal component.

\begin{lemma}[Identification of the flux space]\label{lem:XD-identification}
Assume (P1). Then, with equivalent norms,
\begin{equation}
X_D
=
\left\{
q : q n\in H^{1/2}(\partial D)^3,\;
\int_{\partial D_j} q\,\dd S=0,\;
j=1,\ldots,N
\right\},
\label{eq:XD-general}
\end{equation}
where $q n$ denotes the vector field with value $q(x)n(x)$ on $\partial D$ and
$n$ is the outward unit normal.
\end{lemma}

\begin{proof}
Write $\widetilde X$ for the right-hand side, normed by
$\|q\|_{\widetilde X}:=\|q n\|_{H^{1/2}(\partial D)^3}$.
We first prove $X_D\subset\widetilde X$. Let $q=M\cdot n$ with $M\in A_D$. By (P1),
$M\in H^1(D)^3$, so $M|_{\partial D}\in H^{1/2}(\partial D)^3$. Since
$M\times n=0$, the trace is purely normal, $M|_{\partial D}=(M\cdot n)n=q n$,
whence $q n\in H^{1/2}(\partial D)^3$ and
$\|q n\|_{H^{1/2}}\le C\|M\|_{H^1}\le C'\|M\|_{A_D}$.
The zero-mean condition is $\int_{\partial D_j}M\cdot n=\int_{D_j}\diver M=0$.
Taking the infimum over liftings gives $\|q\|_{\widetilde X}\le C'\|q\|_{X_D}$.

Conversely, let $q\in\widetilde X$. We show that $q\in X_D$. Since
$q n\in H^{1/2}(\partial D)^3$, the trace theorem on the Lipschitz domain $D$
provides $W\in H^1(D)^3$ with $W|_{\partial D}=q n$ and
$\|W\|_{H^1(D)}\le C\|q n\|_{H^{1/2}}$. Set $f:=\diver W\in L^2(D)$. Using
$(q n)\cdot n=q$,
\[
\int_{D_j} f\,\dd x
= \int_{\partial D_j} W\cdot n\,\dd S
= \int_{\partial D_j} q\,\dd S = 0,
\qquad j=1,\ldots,N,
\]
so $f$ has zero mean on each component. By Bogovskii's right inverse of the
divergence \cite[Theorem~2.4]{BorchersSohr1990}, there is $Z\in H^1_0(D)^3$ with
$\diver Z=f$ and $\|Z\|_{H^1}\le C\|f\|_{L^2}$.
Put $M:=W-Z$. Then $M\in H^1(D)^3$, $\diver M=0$, and, since
$Z|_{\partial D}=0$, $M|_{\partial D}=W|_{\partial D}=q n$; hence
$M\times n=0$ and $M\cdot n=q$. Thus $M\in A_D$ and $q\in X_D$, with
$\|q\|_{X_D}\le\|M\|_{A_D}\le C\|q\|_{\widetilde X}$.
\end{proof}

\begin{remark}\label{rem:two-cases}
If $\partial D$ is of class $C^{1,1}$, then $n$ is Lipschitz, hence a
multiplier of $H^{1/2}(\partial D)$. In this case $q n\in H^{1/2}(\partial D)^3$
if and only if $q\in H^{1/2}(\partial D)$, and \eqref{eq:XD-general} reduces to
\[
X_D = H^{1/2}_\diamond(\partial D)
:= \left\{ g\in H^{1/2}(\partial D) :
\int_{\partial D_j} g\,\dd S=0,\; j=1,\ldots,N \right\}.
\]
\end{remark}

\subsection{Sampling functions and the range characterization}

For $a\in B_R$ and $p\in\R^3\setminus\{0\}$, define the dipole potential
\[
d_{a,p}(x) := p\cdot\nabla_x\Phi(x,a)
= -\frac{p\cdot(x-a)}{4\pi|x-a|^3}.
\]
The corresponding sampling function on $\GammaR$ is
\begin{equation}
\varphi_{a,p} := -\nabla d_{a,p}|_{\GammaR}.
\label{eq:sampling-function}
\end{equation}

\begin{lemma}[Range characterization]\label{lem:range-char}
Assume (P1)--(P2) and that $B_R\setminus\overline D$ is connected. Then, for
$a\in B_R\setminus\partial D$,
\begin{equation}
a\in D
\;\Longleftrightarrow\;
\varphi_{a,p}\in\ran(G_D).
\label{eq:range-characterization-G}
\end{equation}
\end{lemma}

\begin{proof}
We prove the two implications separately.

\textbf{Necessity.}
Assume that $a\in D$. After relabeling the connected components, suppose
$a\in D_1$. Let $v\in H^1(D_1)$ be the unique solution of
\[
\Delta v=0\quad\text{in }D_1,\qquad
v=d_{a,p}|_{\partial D_1}\quad\text{on }\partial D_1 .
\]
By (P2), we have $v\in H^2(D_1)$. Define the continuous piecewise harmonic
function
\[
\widetilde d(x)=
\begin{cases}
v(x), & x\in D_1,\\
d_{a,p}(x), & x\in\mathbb{R}^3\setminus\overline{D_1}.
\end{cases}
\]
The jump of the normal derivative across $\partial D_1$ is given by
\[
q_1:=\partial_n^+\widetilde d-\partial_n^-\widetilde d
\in L^2(\partial D_1).
\]
Since the traces of $\widetilde d$ coincide on $\partial D_1$, the tangential
derivatives are also continuous, and hence
\[
\nabla d_{a,p}-\nabla v
=
q_1 n
\qquad\text{on }\partial D_1 .
\]
The left-hand side belongs to $H^{1/2}(\partial D_1)^3$ due to the smoothness of $d_{a,p}$ and $v$. Therefore
\[
q_1 n\in H^{1/2}(\partial D_1)^3 .
\]
Moreover,
\[
\int_{\partial D_1}q_1\,\dd S
=
\int_{\partial D_1}\partial_n^+d_{a,p}\,\dd S
-
\int_{\partial D_1}\partial_n^-v\,\dd S
=0,
\]
because $v$ is harmonic in $D_1$ and $d_{a,p}$ is a dipole potential with zero
net flux through any closed surface enclosing $a$. Extending $q_1$ by zero to
the other connected components of $\partial D$, we obtain $q\in X_D$ by
\eqref{eq:XD-general}.

Since $\widetilde d$ is continuous, harmonic away from $\partial D_1$, and
decays at infinity, the standard single-layer representation yields
\[
\widetilde d=S_{\partial D_1}q_1 .
\]
Therefore, on $\Gamma_R$ we have $S_{\partial D}q=d_{a,p}$ and hence
\[
G_Dq
=
-\nabla S_{\partial D}q|_{\Gamma_R}
=
-\nabla d_{a,p}|_{\Gamma_R}
=
\varphi_{a,p}.
\]

\textbf{Sufficiency.}
Conversely, assume that $a\notin\overline D$ and
$\varphi_{a,p}\in\ran(G_D)$. Then there exists $q\in X_D$ such
that
\[
-\nabla S_{\partial D}q
=
-\nabla d_{a,p}
\qquad\text{on }\Gamma_R .
\]
Define
\[
w:=S_{\partial D}q-d_{a,p}.
\]
Then $w$ is harmonic in
\[
\mathcal{O}_a:=\mathbb{R}^3\setminus(\overline D\cup\{a\}),
\]
and satisfies
\[
\nabla w=0\qquad\text{on }\Gamma_R .
\]
Since $w$ decays at infinity, the exterior uniqueness result for the Laplace
equation implies
\[
w=0
\qquad\text{in }\mathbb{R}^3\setminus\overline{B_R}.
\]
The connectedness of $B_R\setminus\overline D$ implies that
$\mathcal{O}_a$ is connected. Hence, by the unique continuation property for
harmonic functions,
\[
S_{\partial D}q=d_{a,p}
\qquad\text{in }\mathcal{O}_a .
\]
This is impossible in a neighbourhood of $a$, because the single-layer
potential $S_{\partial D}q$ is smooth near $a$, whereas $d_{a,p}$ has a
non-removable dipole singularity at $a$. 

\end{proof}

Combining \eqref{eq:range-identity} and \eqref{eq:range-characterization-G}, we
obtain the sampling characterization
\begin{equation}
a\in D
\;\Longleftrightarrow\;
\varphi_{a,p}\in\ran\bigl((\im F_D)^{1/2}\bigr),
\qquad a\in B_R\setminus\partial D.
\label{eq:final-range-characterization}
\end{equation}

\subsection{Support uniqueness}

\begin{theorem}[Recovery of the open support]\label{thm:uniqueness}
Let $D,\widetilde D\Subset B_R$ be bounded open sets, each satisfying the standing
topological assumptions of Section~2 together with (P1)--(P2); in particular, it
suffices that every connected component of $D$ and $\widetilde D$ be of class
$C^{1,1}$ or convex. Assume $B_R\setminus\overline{D}$ and
$B_R\setminus\overline{\widetilde D}$ are connected. Let $\sigma$ and
$\widetilde{\sigma}$ be positive bounded conductivities supported by $D$ and
$\widetilde D$, respectively, and let $F_D$ and $F_{\widetilde D}$ be the
magnetic-dipole near-field operators. If
\[
F_D = F_{\widetilde D},
\]
then the open supports coincide:
\[
D = \widetilde D.
\]
\end{theorem}

\begin{proof}
From $F_D=F_{\widetilde D}$ we get $\im F_D=\im F_{\widetilde D}$. Therefore
\[
\ran\bigl((\im F_D)^{1/2}\bigr)
= \ran\bigl((\im F_{\widetilde D})^{1/2}\bigr).
\]
Using \eqref{eq:final-range-characterization}, for every
$a\in B_R\setminus(\partial D\cup\partial \widetilde D)$,
\[
a\in D
\;\Longleftrightarrow\;
\varphi_{a,p}\in\ran\bigl((\im F_D)^{1/2}\bigr)
\;\Longleftrightarrow\;
\varphi_{a,p}\in\ran\bigl((\im F_{\widetilde D})^{1/2}\bigr)
\;\Longleftrightarrow\;
a\in \widetilde D.
\]
Thus, writing $N:=\partial D\cup\partial \widetilde D$, we have
$D\setminus N = \widetilde D\setminus N$. Each of $D$ and $\widetilde D$ is Lipschitz,
and hence a regular open set, and its boundary has an empty interior, so
$D\setminus N$ is dense in $D$ and $\widetilde D\setminus N$ is dense in
$\widetilde D$. Therefore
\[
\overline D = \overline{D\setminus N}
= \overline{\widetilde D\setminus N} = \overline{\widetilde D},
\]
and consequently
\[
D = \operatorname{int}(\overline D)
= \operatorname{int}(\overline{\widetilde D}) = \widetilde D.
\]
\end{proof}

\section{A numerical factorization algorithm}\label{sect5}

We now outline a finite-dimensional reconstruction method motivated by
\eqref{eq:final-range-characterization}.

\subsection{Finite-dimensional data model}

Let \(x_1,\dots,x_{N}\in\GammaR\) be a set of approximately uniformly
distributed points on the measurement sphere, generated for instance by the
Fibonacci sphere algorithm~\cite{Gonzalez2010}. At each point we place three
magnetic dipoles oriented along the Cartesian axes, giving the
source family
\[
m_{j,k}: x\mapsto e_k\,\delta_{x_j}(x),\qquad
j=1,\dots,N,\; k=1,2,3,
\]
where \(e_k\) is the vector of the \(k\)-th Cartesian unit and \(\delta_{x_j}\) is the
surface Dirac delta at \(x_j\). The incident magnetic field in \(D\) produced by a
point dipole \(m_{j,k}\) is
\[
H_0(x)=-\nabla_x\,\frac{e_k\cdot(x-x_j)}{4\pi|x-x_j|^3},
\]
and the corresponding incident electric field is
\[
E_0(x)=\frac{\ii\omega\mu}{4\pi}\,
\frac{e_k\times(x-x_j)}{|x-x_j|^3}.
\]

The experiment consists of exciting each dipole \(m_{j,k}\) in turn and
recording the induced magnetic field \(H_s\) at every source point
\(x_i\). Ordering the source indices as \(l=(j-1)\cdot3+k\) and the
measurement indices as \(m=(i-1)\cdot3+a\), the $(m,l)$-entry of the near-field
matrix $\mathbf{F}$ is
\[
 F_{ml}
:=
e_a\cdot H_s(x_i),
\]
where \(H_s\) is the induced field produced by the source
\(m_{j,k}\). With \(N\) source points, the matrix
\(\mathbf F\) is of size \(3N\times3N\), with $i,j=1,\cdots,N$ and $k,a=1,2,3$.
In the experiments below, we use
\(N=80\), giving \(240\) sources and receivers.

\paragraph{Semi-positive data matrix.}
 In finite dimensions, consistent with the convention
\(\im T=(T-T^\dagger)/(2\ii)\) used above, we define the Hermitian matrix
\[
\mathbf F_I
:=
\frac{\mathbf F-\mathbf F^\dagger}{2\ii}.
\]
Theoretically, \(\im F_D\) is a semi-positive operator. However, in practice, measurement noise and discretization error may destroy positivity.
We therefore use the regularized positive matrix
\begin{equation}
\mathbf F_{I,\alpha}:=\mathbf F_I+\alpha I, \label{reguF}
\end{equation}
with a special chosen $\alpha>0$.
\paragraph{Sampling vectors.}

For a sampling point \(z\in B_R\) and a direction \(p\in\R^3\setminus\{0\}\),
define the dipole sampling field
\[
\varphi_{z,p}(x):=-\nabla d_{z,p}(x),\qquad
d_{z,p}(x)=p\cdot\nabla_x\Phi(x,z)
=-\frac{p\cdot(x-z)}{4\pi|x-z|^3}.
\]
A direct computation gives
\[
\varphi_{z,p}(x)=\frac{p-3(p\cdot\hat r)\hat r}{4\pi|x-z|^3},
\qquad \hat r=\frac{x-z}{|x-z|}.
\]
The discrete sampling vector \(\bm\phi_{z,p}\in\mathbb C^{3N}\) is formed by
evaluating \(\varphi_{z,p}\) at the measurement points:
\[
(\bm\phi_{z,p})_{(i-1)\cdot3+a}
:=e_a\cdot\varphi_{z,p}(x_i),\qquad
i=1,\dots,N,\; a=1,2,3.
\]

\paragraph{Indicator function.}

The range characterization  suggests the Picard-type criterion \cite{KirschGrinberg2008}. It follows from \eqref{eq:final-range-characterization} that a point 
$z$ belongs to the domain $D$ if and only if
\[
\varphi_{z,p}\in\ran\bigl((\im F_D)^{1/2}\bigr).
\]
For a compact positive operator with eigenpairs \((\lambda_n,\psi_n)\), this is
equivalent to
\[
\sum_n \frac{|(\varphi_{z,p},\psi_n)|^2}{\lambda_n}<\infty.
\]
Thus we define the regularized directional indicator as the normalized Picard
quotient
\[
I_\alpha(z,p)
:=
\frac{\|\bm\phi_{z,p}\|^2}
{\bm\phi_{z,p}^*\,\mathbf F_{I,\alpha}^{-1}\,\bm\phi_{z,p}}.
\]
The numerator \(\|\bm\phi_{z,p}\|^2\) renders the indicator invariant under the
\(z\)-dependent scaling of the sampling vector. The quotient is a regularized Picard indicator. 
In accordance with the range characterization, it is expected to remain relatively large for 
\(z\in D\) and to be small for \(z\notin D\). To reduce directional dependence, we use the three Cartesian
directions \(e_1,e_2,e_3\) and pool the numerators and denominators separately,
\begin{equation}
I_\alpha(z)
:=
\frac{\sum_{k=1}^3\|\bm\phi_{z,e_k}\|^2}
{\sum_{k=1}^3\bm\phi_{z,e_k}^*\,\mathbf F_{I,\alpha}^{-1}\,\bm\phi_{z,e_k}}.
\label{eq:indicator}
\end{equation}
Pooling before dividing is more stable under noise than averaging the
three directional quotients. The support is then reconstructed by thresholding
\(I_\alpha(z)\), which by the range characterization is relatively large for
\(z\in D\) and small for \(z\notin D\), up to resolution and noise effects.

The cost of the indicator is low. The matrix \(\mathbf F_{I,\alpha}\) is
independent of the sampling point, so it is factored once and reused. Each point
then requires only the three solves \(\mathbf F_{I,\alpha}^{-1}\bm\phi_{z,e_k}\)
against this stored factorization and the corresponding inner products. No
forward problem is solved during the reconstruction stage. The eigenvalues of \(\mathbf F_I\) are
not needed for the reconstruction itself. Only the single most negative one
enters through the parameter rule below.

\paragraph{Choice of the regularization parameter.}
The Hermitian matrix \(\mathbf F_I\) is
is positive semidefinite in the continuum limit; any negative eigenvalues arise solely from measurement noise and discretization errors. In \eqref{reguF}, 
 We set
\begin{equation}
\alpha=\max\{0, -\lambda_{\min}(\mathbf F_I)\}+\varepsilon_0,
\qquad \varepsilon_0 = 10^{-9},
\label{eq:alpha-rule}
\end{equation}
where the small additive constant \(\varepsilon_0\) serves as a fixed
regularization floor. This rule guarantees that
$\mathbf F_I+\alpha I$ is positive definite and hence permits a Cholesky factorization. 
The ratio \(\alpha/\lambda_{\max}(\mathbf F_I)\) measures how much
regularization is needed relative to the dominant eigenvalue. A small
ratio (e.g.\ between \(10^{-5}\) and \(10^{-3}\)) confirms that \(\mathbf F_I\) is
nearly semi-positive definite and the regularization is genuinely a minor
correction, as predicted by the theory. This quantity will be reported for every
numerical example below.

\begin{algorithm}[htpb]
\caption{Reconstruction via the factorization range identity}
\label{alg:reconstruction}
\SetAlgoLined
\KwIn{near-field matrix $\mathbf F\in\mathbb C^{3N\times 3N}$, sampling grid $\{z_j\}\subset B_R$}
\KwOut{normalized indicator $\tilde I_\alpha(z)$}

$\mathbf F_I \gets (\mathbf F-\mathbf F^\dagger)/(2\ii)$\;
$\alpha \gets \max\{0, -\lambda_{\min}(\mathbf F_I)\}+\varepsilon_0$\;
$\mathbf F_{I,\alpha} \gets \mathbf F_I + \alpha I$\;
Compute the Cholesky factorization of $\mathbf F_{I,\alpha}$\;

\For{each sampling point $z_j$}{
    Compute the sampling vector $\bm\phi_{z_j,e_k}$ for $k=1,2,3$\;
    Solve $\mathbf F_{I,\alpha}\,\mathbf x_k = \bm\phi_{z_j,e_k}$ for each $k$ via the stored factorization\;
    $I_\alpha(z_j) \gets \dfrac{\sum_{k=1}^3\|\bm\phi_{z_j,e_k}\|^2}
                             {\sum_{k=1}^3\bm\phi_{z_j,e_k}^*\,\mathbf x_k}$\;
}

Normalize $\tilde I_\alpha \gets I_\alpha/\max I_\alpha$.
\end{algorithm}

\subsection{Numerical experiments}

\paragraph{Settings.}
The synthetic data are generated by an independent edge (N\'ed\'elec) finite
element solver for the full eddy-current system~\eqref{eq:full-eddy}, with
angular frequency \(\omega=2\pi\times10^{6}\), constant permeability
\(\mu=4\pi\times10^{-7}\), and conductivity \(\sigma=10^{-1}\) inside 
the inclusions. The measurement sphere has radius
\(R=1.5\). Sources and receivers are point magnetic dipoles at
\(N=80\) Fibonacci-sphere points as described above, with three
Cartesian directions per point, giving a total of \(3N=240\) excitations and
measurements. The near-field matrix \(\mathbf F\in\mathbb C^{240\times240}\) collects the
measured induced-field components.

To assess robustness, we perturb relative multiplicative noise of level \(\delta\) independently to the
real and imaginary parts of each datum according to
\[\re F_{ml}\mapsto(1+\delta\xi_{ml,1})\re F_{ml},\qquad\im F_{ml}\mapsto(1+\delta\xi_{ml,2})\im F_{ml}\]
with \(\xi_{ml,1},\xi_{ml,2}\overset{\text{i.i.d.}}{\sim} U(-1,1)\). 
The indicators are evaluated on a \(140\times140\) grid on the plane \(z=0\).

\paragraph{Example 1: two inclusions.}
This example verifies that the method accurately localizes multiple
inclusions.
Two cubes of side \(0.2\) are centered at \((0.3,0.3,0)\) and
\((-0.2,-0.4,0)\). Figure~\ref{fig:spectrum2sq} displays the eigenvalue magnitudes
\(|\lambda_j|\) of \(\mathbf F_I\) in descending order on a logarithmic
scale. Of the \(240\) eigenvalues, only \(32\) have small negative values. The most
negative eigenvalue is \(\lambda_{\min}\approx-3.30\times10^{-9}\), while the
largest positive eigenvalue is \(\lambda_{\max}\approx5.04\times10^{-5}\). 
This provides direct numerical evidence that \(\mathbf F_I\) inherits the
theoretical semi-positivity of \(\im F_D\) up to a negligible numerical perturbation. 
Figure~\ref{fig:recon2sqN0} shows the noise-free
reconstruction. The indicator forms two sharply separated maxima located at the
true centers, which is consistent with the range characterization.

\begin{figure}[htpb]
\centering
\includegraphics[width=0.55\textwidth]{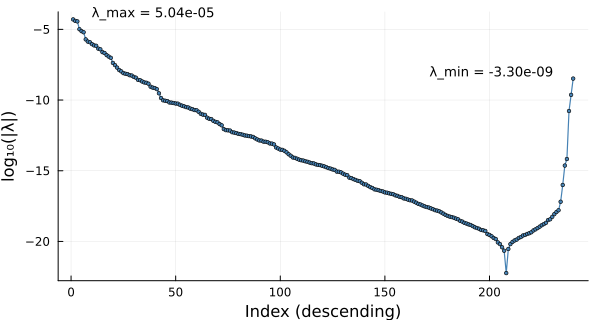}
\caption{Eigenvalue magnitudes of \(\mathbf F_I\) for the noise-free
two-inclusion example, plotted in descending order on a \(\log_{10}\) scale.}
\label{fig:spectrum2sq}
\end{figure}

\begin{figure}[htpb]
\centering
\begin{subfigure}{0.42\textwidth}
\centering
\includegraphics[width=\textwidth]{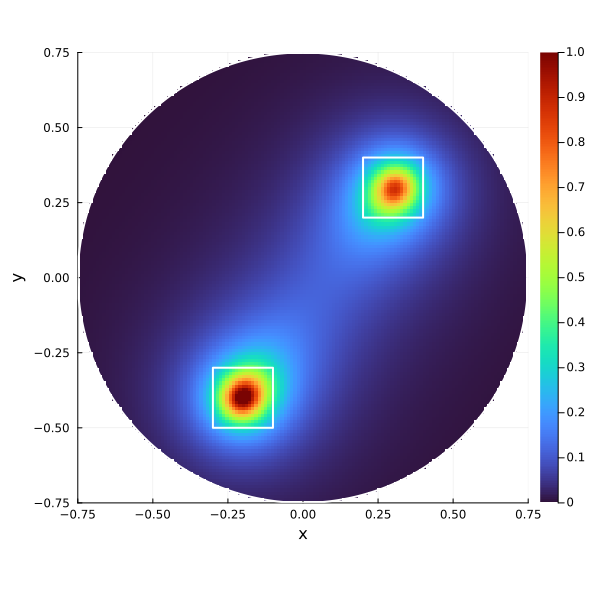}
\caption{}
\end{subfigure}
\hfill
\begin{subfigure}{0.47\textwidth}
\centering
\includegraphics[width=\textwidth]{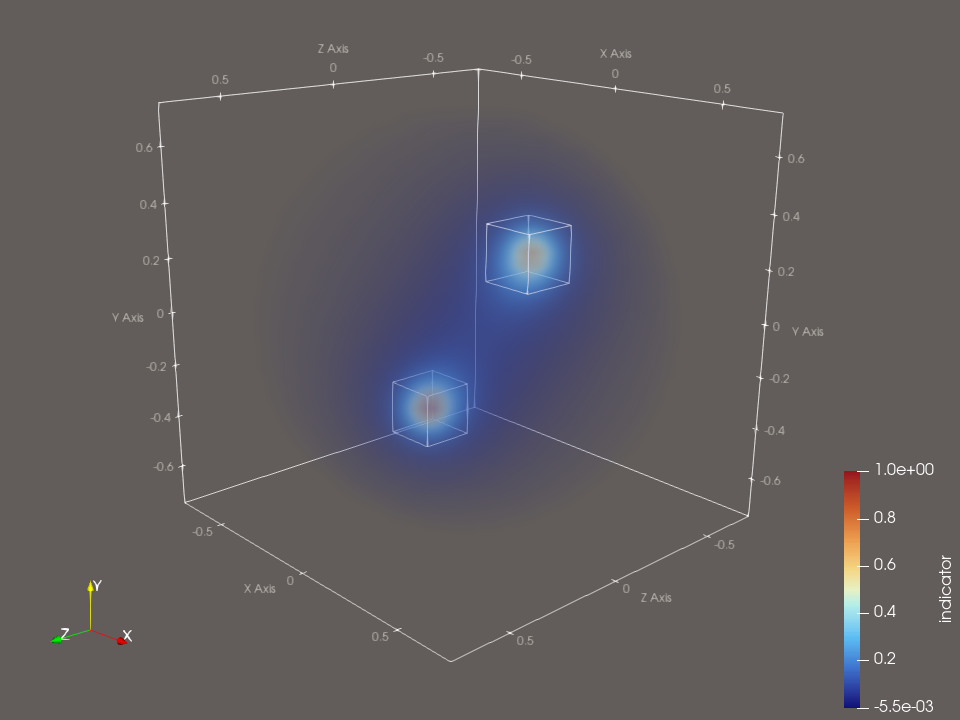}
\caption{}
\end{subfigure}
\caption{Reconstruction of two cubic inclusions from noise-free data.
The white squares indicate the true conductors. The data-driven parameter is
\(\alpha/\lambda_{\max}=8.54\times10^{-5}\). Left: indicator on cross-section \(z=0\). Right: volume rendering of the indicator.}
\label{fig:recon2sqN0}
\end{figure}

\paragraph{Example 2: robustness to noise.}
This example tests the stability of the method under increasing
measurement noise. As the noise level grows, \(\mathbf F_I\) loses
positive definiteness and \(\alpha/\lambda_{\max}\) increases.
Figure~\ref{fig:noise} shows the reconstructions for noise levels
\(\delta=0,1,2,5,10\%\) under the rule \eqref{eq:alpha-rule} with
the fixed floor \(\varepsilon_0=10^{-9}\). The
regularization parameter adapts automatically, with the relative value
\(\alpha/\lambda_{\max}\) increasing from \(8.54\times10^{-5}\) to
\(8.88\times10^{-3}\) as documented in Table~\ref{tab:noise}. 
This monotonic increase quantifies the progressive loss of semi-positive definiteness of \(\mathbf F_I\) under
measurement noise. Stronger noise pushes the most negative eigenvalue further below
zero, requiring proportionally more regularization to restore positivity. However, the two
inclusions are still individually recovered at high noise level.

\begin{table}[htpb]
\centering
\begin{tabular}{cc}
\hline
noise \(\delta\) & \(\alpha/\lambda_{\max}\) \\
\hline
\(0\%\)  & \(8.54\times10^{-5}\) \\
\(1\%\)  & \(9.74\times10^{-4}\) \\
\(2\%\)  & \(1.83\times10^{-3}\) \\
\(5\%\)  & \(4.78\times10^{-3}\) \\
\(10\%\) & \(8.88\times10^{-3}\) \\
\hline
\end{tabular}
\caption{Two-inclusion example: the data-driven regularization parameter
\eqref{eq:alpha-rule} with \(\varepsilon_0=10^{-9}\) as a function of the noise level.}
\label{tab:noise}
\end{table}

\begin{figure}[htpb]
\centering
\includegraphics[width=0.19\textwidth]{figures/indicator_2sq_N0.png}\hfill
\includegraphics[width=0.19\textwidth]{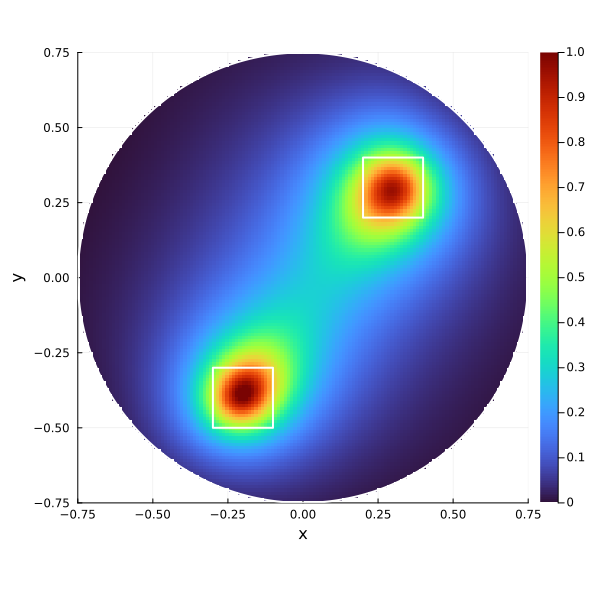}\hfill
\includegraphics[width=0.19\textwidth]{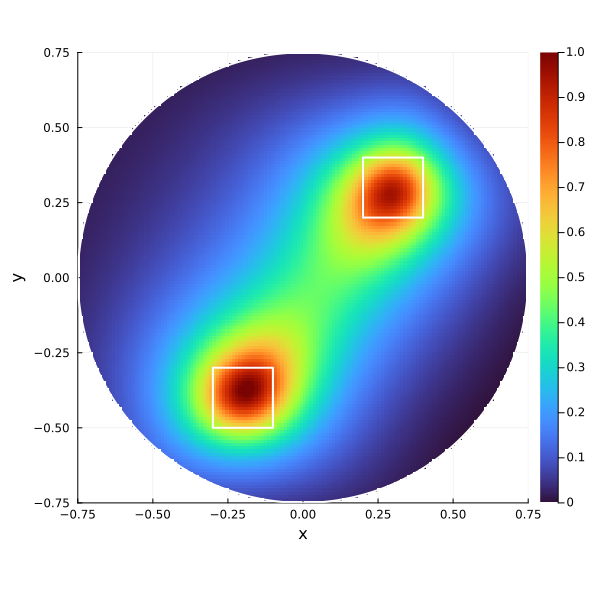}\hfill
\includegraphics[width=0.19\textwidth]{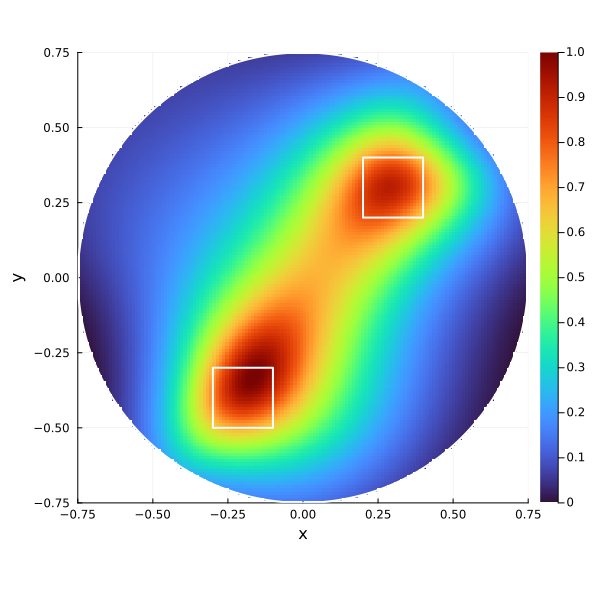}\hfill
\includegraphics[width=0.19\textwidth]{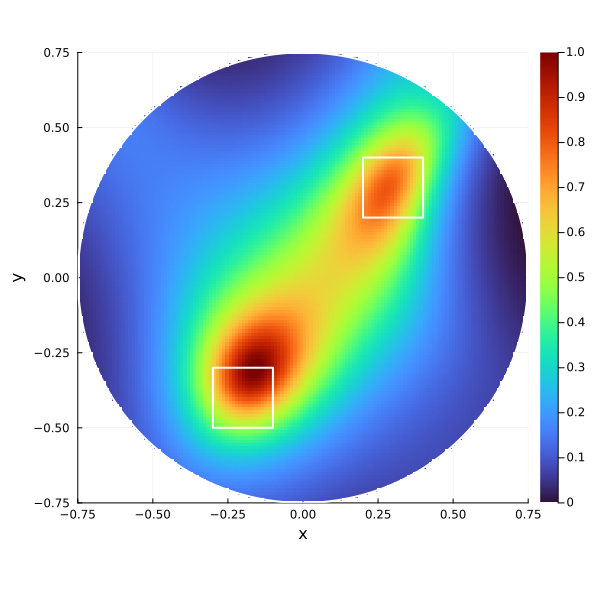}\\
\centering
\includegraphics[width=0.19\textwidth]{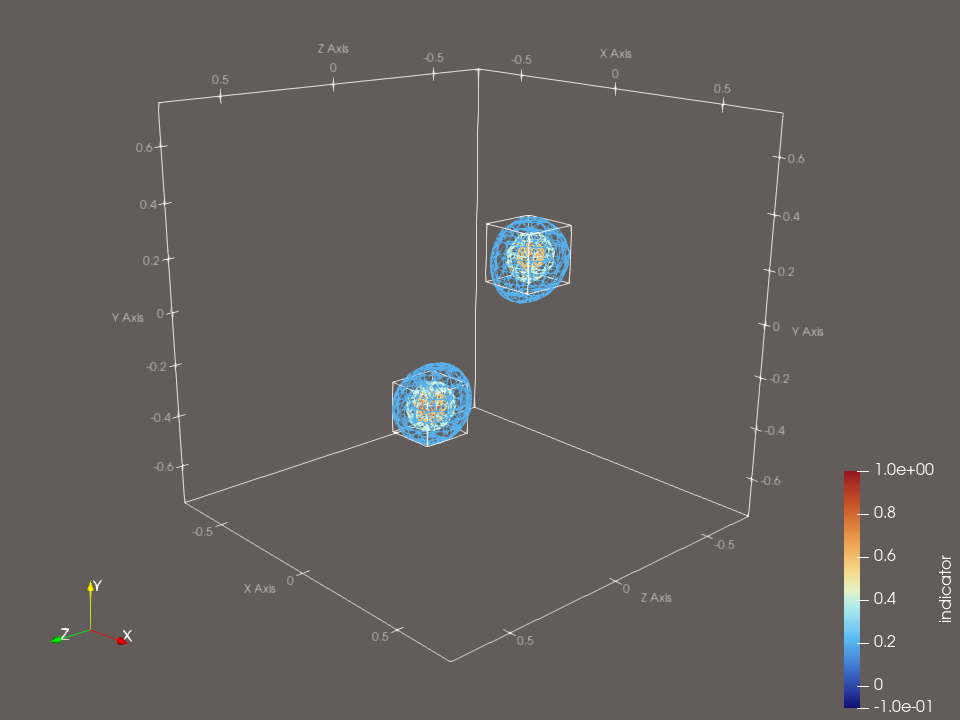}\hfill
\includegraphics[width=0.19\textwidth]{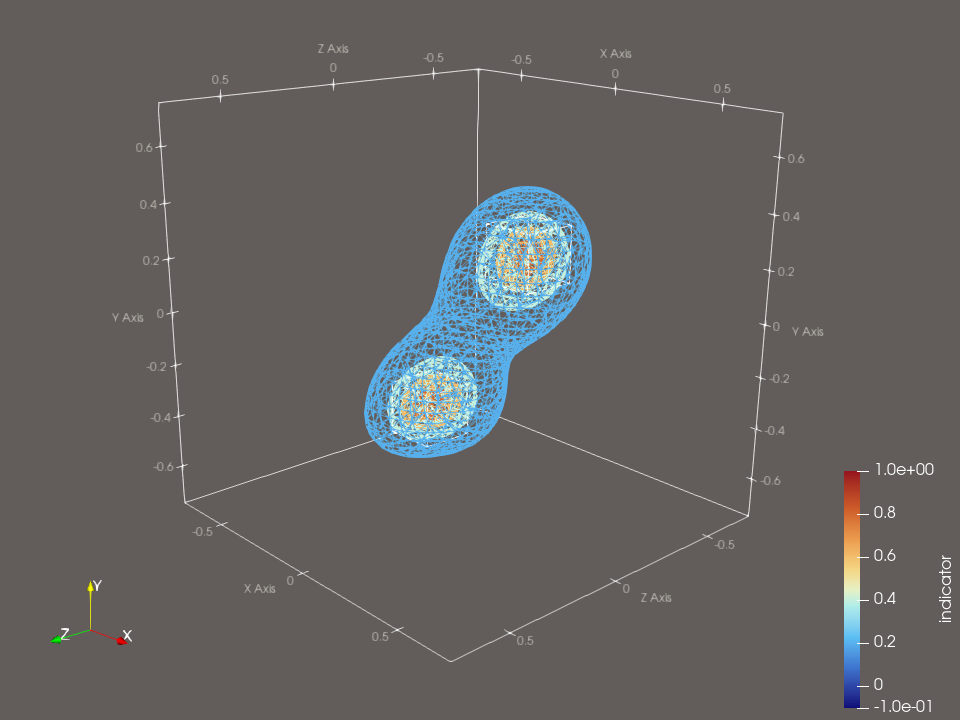}\hfill
\includegraphics[width=0.19\textwidth]{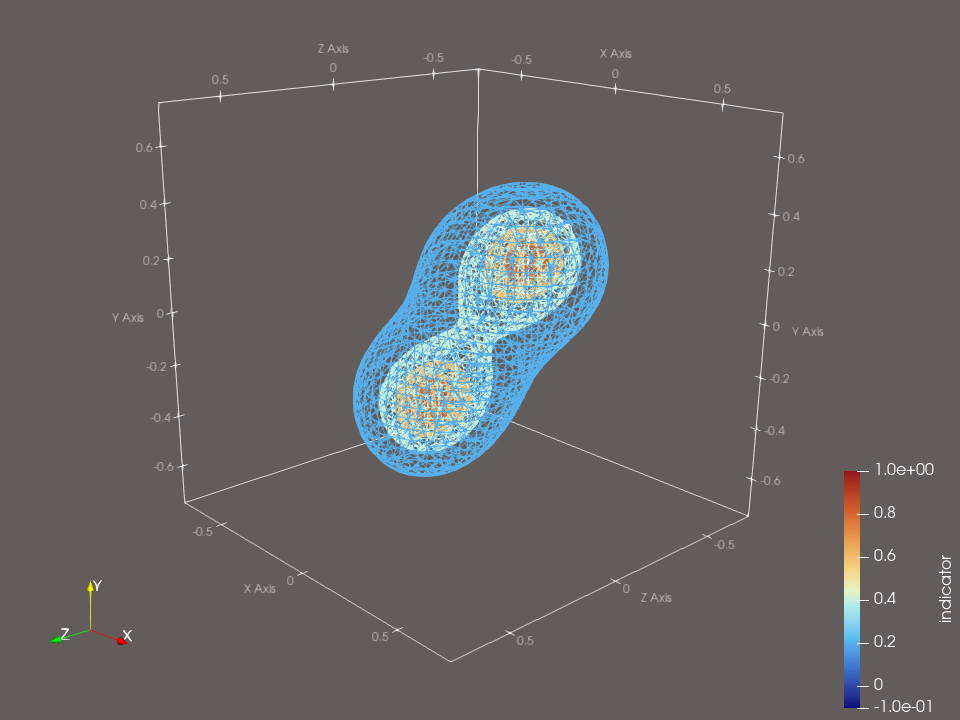}\hfill
\includegraphics[width=0.19\textwidth]{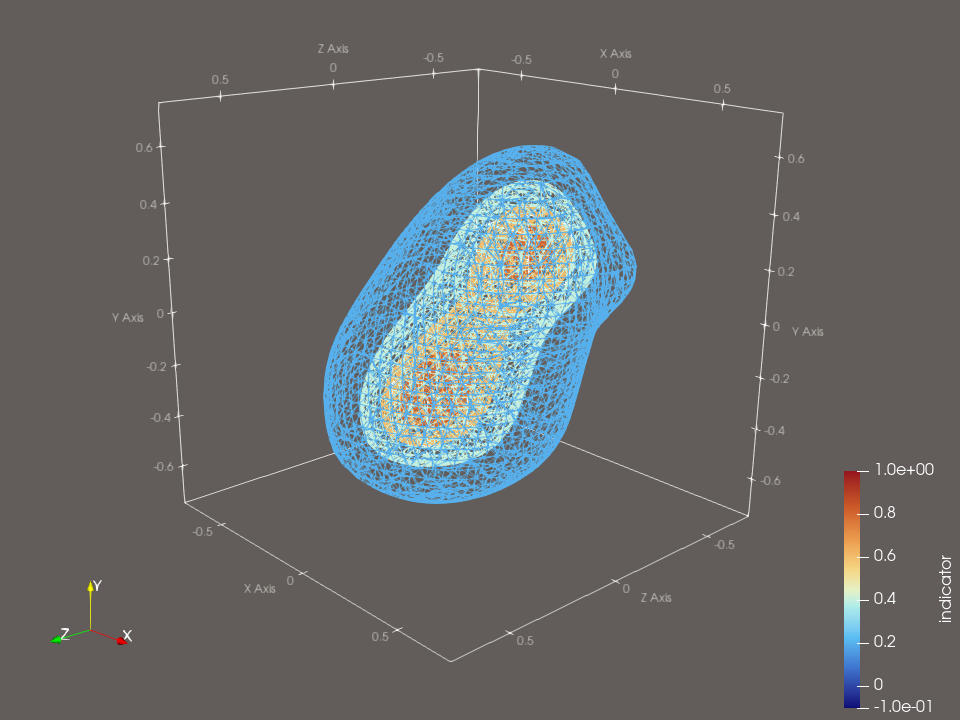}\hfill
\includegraphics[width=0.19\textwidth]{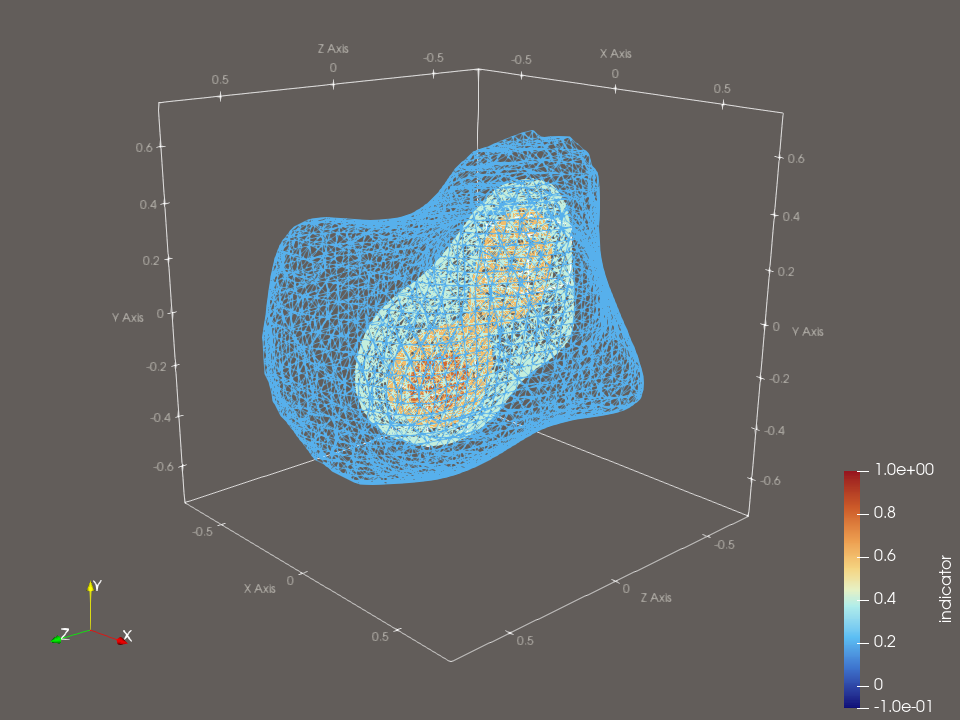}
\caption{Reconstructions of the two inclusions at noise levels
\(\delta=0,1,2,5,10\%\) (left to right) under the single rule
\eqref{eq:alpha-rule} with fixed \(\varepsilon_0=10^{-9}\). 
Top: indicators on cross-section \(z=0\). Bottom: indicator isosurfaces at levels $0.2$,$0.4$,$0.6$,$0.8$.}
\label{fig:noise}
\end{figure}

\paragraph{Example 3: a non-convex inclusion.}
Figure~\ref{fig:reconL} shows the reconstruction of an L-shaped conductor placed
on $xOy$ plane with thickness $0.2$. The inclusion is non-convex and has reentrant
edges, so it is neither convex nor of class $C^{1,1}$ and therefore lies outside
the geometry class covered by Theorem~\ref{thm:uniqueness}. 
However, both arms and the corner are recovered, showing that the
method captures non-convex supports without any convexity prior.

\begin{figure}[htpb]
\centering
\begin{subfigure}{0.42\textwidth}
\centering
\includegraphics[width=\textwidth]{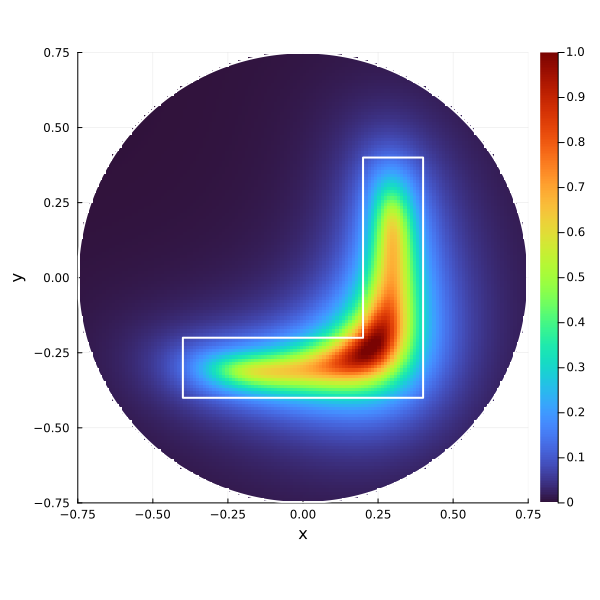}
\caption{}
\end{subfigure}
\hfill
\begin{subfigure}{0.47\textwidth}
\centering
\includegraphics[width=\textwidth]{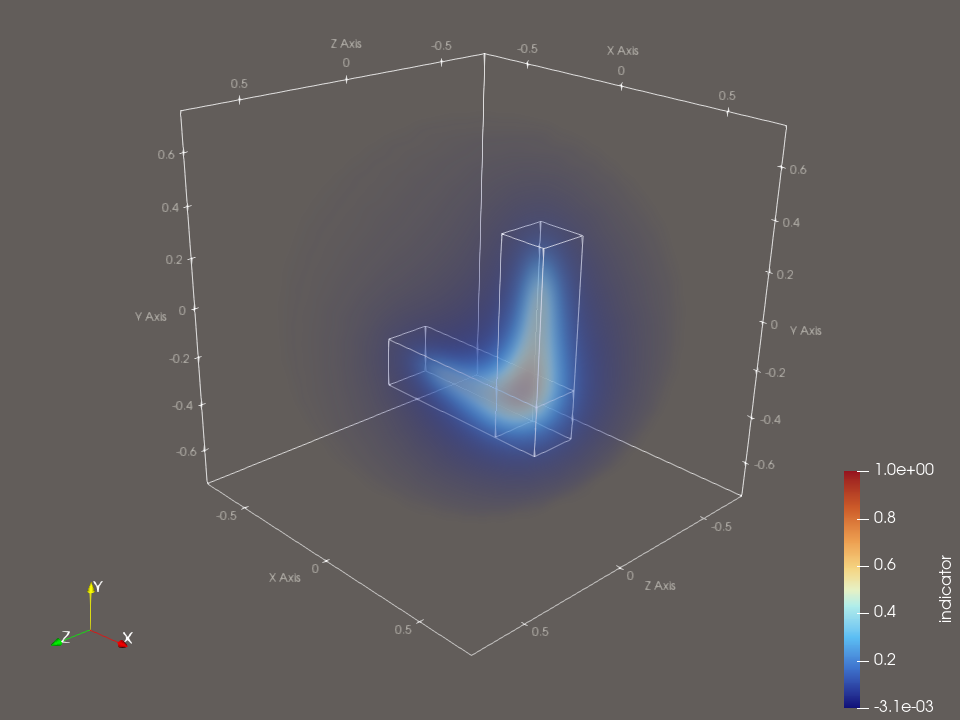}
\caption{}
\end{subfigure}
\caption{Reconstruction of a non-convex L-shaped conductor from
noise-free data. The parameter is
\(\alpha/\lambda_{\max}=2.22\times10^{-4}\). Left: indicator on cross-section \(z=0\). Right: volume rendering of the indicator.}
\label{fig:reconL}
\end{figure}

\section{Conclusion}\label{sect6}

We have developed a factorization framework for support recovery in a three-dimensional 
time-harmonic eddy-current model of magnetic induction tomography with constant
magnetic permeability. The central modeling step is the representation of the
induced current \(J=\sigma E\) by a current vector potential \(M\) with
\(J=\curl M\), \(\diver M=0\) and \(M\times n=0\) on \(\partial D\), which
reduces the full-space problem to an interior variational problem on the unknown
conductor \(D\) whose exterior signature is carried by the single normal flux
\(M\cdot n\). This leads to the physical factorization
\(F_D=G_DB_DG_D^*\). Using 
the Riesz map of \(X_D\) to distinguish the anti-dual adjoint \(G_D^*\) from the
Hilbert adjoint \(G_D^\dagger\), we obtained 
\(\im F_D=G_DP_DG_D^\dagger\) where \(P_D\) is coercive. Douglas' theorem then yields the range identity
\(\ran((\im F_D)^{1/2})=\ran(G_D)\). This identity holds for bounded Lipschitz
inclusions under mild topological assumptions. For inclusions whose connected
components are of class \(C^{1,1}\) or convex, we deduced that the magnetic near-field operator
determines the conductive support uniquely.

On the computational side we proposed a non-iterative spectral indicator built
from the Hermitian part \(\mathbf F_I\) of the
data matrix, where only one small linear solve per sampling point is needed. The
regularization is driven by the data through the rule
\(\alpha=\max\{0, -\lambda_{\min}(\mathbf F_I)\}+\varepsilon_0\) with the single
fixed floor \(\varepsilon_0=10^{-9}\). Numerical experiments with two cubic inclusions and a
non-convex L-shaped conductor 
confirmed that the method locates and separates
the inclusions successfully, and that the reconstruction degrades gracefully as the noise
increases.

\bibliographystyle{abbrv}
\bibliography{references}

@article{ABDG1998,
    author = {Amrouche, C. and Bernardi, C. and Dauge, M. and Girault, V.},
    title = {Vector potentials in three-dimensional non-smooth domains},
    journal = {Mathematical Methods in the Applied Sciences},
    volume = {21},
    number = {9},
    pages = {823-864},
    doi = {https://doi.org/10.1002/(SICI)1099-1476(199806)21:9<823::AID-MMA976>3.0.CO;2-B},
    url = {https://onlinelibrary.wiley.com/doi/abs/10.1002/%28SICI%291099-1476%28199806%2921%3A9%3C823%3A%3AAID-MMA976%3E3.0.CO%3B2-B},
    eprint = {https://onlinelibrary.wiley.com/doi/pdf/10.1002/%28SICI%291099-1476%28199806%2921%3A9%3C823%3A%3AAID-MMA976%3E3.0.CO%3B2-B},
    year = {1998}
}

@article{Gonzalez2010,
    author  = {Gonz{\'a}lez, {\'A}lvaro},
    title   = {Measurement of areas on a sphere using {Fibonacci} and latitude--longitude lattices},
    journal = {Mathematical Geosciences},
    year    = {2010},
    volume  = {42},
    number  = {1},
    pages   = {49--64},
    doi     = {10.1007/s11004-009-9257-x}
}

@article{AmmariChenChenVolkovWang2013,
    title = {Target detection and characterization from electromagnetic induction data},
    journal = {Journal de Mathématiques Pures et Appliquées},
    volume = {101},
    number = {1},
    pages = {54-75},
    year = {2014},
    issn = {0021-7824},
    doi = {https://doi.org/10.1016/j.matpur.2013.05.002},
    url = {https://www.sciencedirect.com/science/article/pii/S0021782413000792},
    author = {Habib Ammari and Junqing Chen and Zhiming Chen and Josselin Garnier and Darko Volkov}
}

@article{ArnoldHarrach2013,
    doi = {10.1088/0266-5611/29/9/095004},
    url = {https://doi.org/10.1088/0266-5611/29/9/095004},
    year = {2013},
    month = {aug},
    publisher = {IOP Publishing},
    volume = {29},
    number = {9},
    pages = {095004},
    author = {Arnold, L and Harrach, B},
    title = {Unique shape detection in transient eddy current problems},
    journal = {Inverse Problems}
}

@article{ChenLiangZou2020,
author = {Chen, Junqing and Liang, Ying and Zou, Jun},
title = {Mathematical and Numerical Study of a Three-Dimensional Inverse Eddy Current Problem},
journal = {SIAM Journal on Applied Mathematics},
volume = {80},
number = {3},
pages = {1467-1492},
year = {2020},
doi = {10.1137/19M1282866},
URL = { 
        https://doi.org/10.1137/19M1282866
},
eprint = { 
        https://doi.org/10.1137/19M1282866
}
}

@article{ChenLong2024,
    author  = {Chen, Junqing and Long, Zehao},
    title   = {An Iterative Method for the Inverse Eddy Current Problem with Total Variation Regularization},
    journal = {Journal of Scientific Computing},
    year    = {2024},
    volume  = {99},
    number  = {2},
    pages   = {38},
    doi     = {10.1007/s10915-024-02501-9}
}

@article{ColtonKirsch1996,
    doi = {10.1088/0266-5611/12/4/003},
    url = {https://doi.org/10.1088/0266-5611/12/4/003},
    year = {1996},
    month = {aug},
    publisher = {},
    volume = {12},
    number = {4},
    pages = {383},
    author = {David Colton and Andreas Kirsch},
    title = {A simple method for solving inverse scattering problems in the resonance region},
    journal = {Inverse Problems}
}

@article{Douglas1966,
    ISSN = {00029939, 10886826},
    URL = {http://www.jstor.org/stable/2035178},
    author = {R. G. Douglas},
    journal = {Proceedings of the American Mathematical Society},
    number = {2},
    pages = {413--415},
    publisher = {American Mathematical Society},
    title = {On Majorization, Factorization, and Range Inclusion of Operators on Hilbert Space},
    urldate = {2026-07-09},
    volume = {17},
    year = {1966}
}

@article{GebauerHankeSchneider2008,
    doi = {10.1088/0266-5611/24/1/015007},
    url = {https://doi.org/10.1088/0266-5611/24/1/015007},
    year = {2007},
    month = {dec},
    publisher = {},
    volume = {24},
    number = {1},
    pages = {015007},
    author = {Gebauer, Bastian and Hanke, Martin and Schneider, Christoph},
    title = {Sampling methods for low-frequency electromagnetic imaging},
    journal = {Inverse Problems}
}

@article{Griffiths2001,
    doi = {10.1088/0957-0233/12/8/319},
    url = {https://doi.org/10.1088/0957-0233/12/8/319},
    year = {2001},
    month = {aug},
    publisher = {},
    volume = {12},
    number = {8},
    pages = {1126},
    author = {H Griffiths},
    title = {Magnetic induction tomography},
    journal = {Measurement Science and Technology}
}

@book{Grisvard1985,
    author = {Grisvard, Pierre},
    title = {Elliptic Problems in Nonsmooth Domains},
    publisher = {Society for Industrial and Applied Mathematics},
    year = {2011},
    doi = {10.1137/1.9781611972030},
    address = {},
    edition   = {},
    URL = {https://epubs.siam.org/doi/abs/10.1137/1.9781611972030},
    eprint = {https://epubs.siam.org/doi/pdf/10.1137/1.9781611972030}
}

@article{HaddarRiahi2021,
    doi = {10.1088/1361-6420/ac1c50},
    url = {https://doi.org/10.1088/1361-6420/ac1c50},
    year = {2021},
    month = {aug},
    publisher = {IOP Publishing},
    volume = {37},
    number = {10},
    pages = {105002},
    author = {Haddar, Houssem and Riahi, Mohamed Kamel},
    title = {Near-field linear sampling method for axisymmetric eddy current tomography},
    journal = {Inverse Problems}
}

@article{Kirsch2004,
    doi = {10.1088/0266-5611/20/6/S08},
    url = {https://doi.org/10.1088/0266-5611/20/6/S08},
    year = {2004},
    month = {nov},
    publisher = {},
    volume = {20},
    number = {6},
    pages = {S117},
    author = {Andreas Kirsch},
    title = {The factorization method for Maxwell's equations},
    journal = {Inverse Problems}
}

@book{KirschGrinberg2008,
    author = {Kirsch, Andreas and Grinberg, Natalia},
    title = {The Factorization Method for Inverse Problems},
    publisher = {Oxford University Press},
    year = {2007},
    month = {12},
    isbn = {9780199213535},
    doi = {10.1093/acprof:oso/9780199213535.001.0001},
    url = {https://doi.org/10.1093/acprof:oso/9780199213535.001.0001},
}

@book{Monk2003,
    author = {Monk, Peter},
    title = {Finite Element Methods for Maxwell's Equations},
    publisher = {Oxford University Press},
    year = {2003},
    month = {04},
    isbn = {9780198508885},
    doi = {10.1093/acprof:oso/9780198508885.001.0001},
    url = {https://doi.org/10.1093/acprof:oso/9780198508885.001.0001},
}

@article{ScharfetterCasanasRosell2003,
    author={Scharfetter, H. and Casanas, R. and Rosell, J.},
    journal={IEEE Transactions on Biomedical Engineering}, 
    title={Biological tissue characterization by magnetic induction spectroscopy (MIS): requirements and limitations}, 
    year={2003},
    volume={50},
    number={7},
    pages={870-880},
    doi={10.1109/TBME.2003.813533}
}

@article{TamburrinoRubinacci2002,
    doi = {10.1088/0266-5611/18/6/323},
    url = {https://doi.org/10.1088/0266-5611/18/6/323},
    year = {2002},
    month = {nov},
    publisher = {},
    volume = {18},
    number = {6},
    pages = {1809},
    author = {A Tamburrino and G Rubinacci},
    title = {A new non-iterative inversion method for electrical resistance tomography},
    journal = {Inverse Problems}
}

@article{TamburrinoPiscitelliZhou2021,
    doi = {10.1088/1361-6420/ac156c},
    url = {https://doi.org/10.1088/1361-6420/ac156c},
    year = {2021},
    month = {aug},
    publisher = {IOP Publishing},
    volume = {37},
    number = {9},
    pages = {095003},
    author = {Tamburrino, Antonello and Piscitelli, Gianpaolo and Zhou, Zhengfang},
    title = {The monotonicity principle for magnetic induction tomography},
    journal = {Inverse Problems}
}

@article{TamburrinoCorboPiscitelli2026,
    doi = {10.1088/1361-6420/ae4823},
    url = {https://doi.org/10.1088/1361-6420/ae4823},
    year = {2026},
    month = {mar},
    publisher = {IOP Publishing},
    volume = {42},
    number = {3},
    pages = {035001},
    author = {Tamburrino, Antonello and Corbo Esposito, Antonio and Piscitelli, Gianpaolo},
    title = {Monotonicity of the Laplace transform for tomography in dissipative systems},
    journal = {Inverse Problems}
}

@article{BorchersSohr1990,
    author = {W. Borchers and H. Sohr},
    title = {{On the equations rot v=g and div u=f with zero boundary conditions}},
    volume = {19},
    journal = {Hokkaido Mathematical Journal},
    number = {1},
    publisher = {Hokkaido University, Department of Mathematics},
    pages = {67 -- 87},
    year = {1990},
    doi = {10.14492/hokmj/1381517172},
    URL = {https://doi.org/10.14492/hokmj/1381517172}
}

@article{Bruhl2001,
author = {Br{\"u}hl, Martin},
title = {Explicit Characterization of Inclusions in Electrical Impedance Tomography},
journal = {SIAM Journal on Mathematical Analysis},
volume = {32},
number = {6},
pages = {1327-1341},
year = {2001},
doi = {10.1137/S003614100036656X},
URL = { 
        https://doi.org/10.1137/S003614100036656X
},
eprint = { 
        https://doi.org/10.1137/S003614100036656X
}
}

@article{Kirsch1998,
  author  = {Kirsch, Andreas},
  title   = {Characterization of the Shape of a Scattering Obstacle Using the Spectral Data of the Far Field Operator},
  journal = {Inverse Problems},
  volume  = {14},
  number  = {6},
  pages   = {1489--1512},
  year    = {1998},
  doi     = {10.1088/0266-5611/14/6/009}
}

@article{ColtonHaddarMonk2003,
  author  = {Colton, David and Haddar, Houssem and Monk, Peter},
  title   = {The Linear Sampling Method for Solving the Electromagnetic Inverse Scattering Problem},
  journal = {SIAM Journal on Scientific Computing},
  volume  = {24},
  number  = {3},
  pages   = {719--731},
  year    = {2003},
  doi     = {10.1137/S1064827501390467}
}

@article{ItoJinZou2013,
  author  = {Ito, Kazufumi and Jin, Bangti and Zou, Jun},
  title   = {A Direct Sampling Method for Inverse Electromagnetic Medium Scattering},
  journal = {Inverse Problems},
  volume  = {29},
  number  = {9},
  pages   = {095018},
  year    = {2013},
  doi     = {10.1088/0266-5611/29/9/095018}
}

@article{ChowItoZou2014,
  author  = {Chow, Yat Tin and Ito, Kazufumi and Zou, Jun},
  title   = {A Direct Sampling Method for Electrical Impedance Tomography},
  journal = {Inverse Problems},
  volume  = {30},
  number  = {9},
  pages   = {095003},
  year    = {2014},
  doi     = {10.1088/0266-5611/30/9/095003}
}

@article{Chen_2013,
doi = {10.1088/0266-5611/29/8/085005},
url = {https://doi.org/10.1088/0266-5611/29/8/085005},
year = {2013},
month = {jul},
publisher = {IOP Publishing},
volume = {29},
number = {8},
pages = {085005},
author = {Chen, Junqing and Chen, Zhiming and Huang, Guanghui},
title = {Reverse time migration for extended obstacles: acoustic waves},
journal = {Inverse Problems}
}

@article{ChenJiang,
    doi={10.1088/1361-6420/aea2c2},
	author={Chen, Junqing and Jiang, Chengzhe},
	title={A direct sampling method for magnetic induction tomography},
	journal={Inverse Problems},
	year={2026}
}
\end{document}